\documentclass{amsart}
\usepackage{amsthm,amssymb,amsmath}
\usepackage{mathtools}
\usepackage[margin=1.3in]{geometry}
\usepackage{xcolor} \definecolor{darkblue}{rgb}{0,0,0.6}
\usepackage[breaklinks, pdftex, ocgcolorlinks, citecolor=darkblue,
filecolor=darkblue, linkcolor=darkblue, urlcolor=black]{hyperref}
\usepackage[capitalize]{cleveref}
\usepackage{enumerate}
\usepackage{color}
\usepackage[matrix,arrow,curve]{xy}
\usepackage{tikz-cd}

\newtheorem{theorem}{Theorem}[section]
\newtheorem{lemma}[theorem]{Lemma}

\newtheorem{corollary}[theorem]{Corollary}

\theoremstyle{definition}

\theoremstyle{remark}
\newtheorem{remark}[theorem]{Remark}
\newtheorem{example}[theorem]{Example}

\newcommand{\wh}{\widehat}
\newcommand{\wt}{\widetilde}
\newcommand{\ol}{\overline}
\newcommand{\ev}{\operatorname{ev}\!{}}
\newcommand{\sq}{\mathbin{\square}}

\DeclareMathOperator{\coker}{coker}
\DeclareMathOperator{\Hom}{Hom}
\DeclareMathOperator{\id}{id}
\DeclareMathOperator{\im}{im}
\DeclareMathOperator{\Ext}{Ext}
\DeclareMathOperator{\Tor}{Tor}
\DeclareMathOperator{\Her}{Her}
\DeclareMathOperator{\Aut}{Aut}
\DeclareMathOperator{\tr}{tr}

\title{Homotopy classification of $4$-manifolds with fundamental group
  $\mathbb{Z}/p \rtimes \mathbb{Z}$}

\author{Kamen V. Pavlov}
\address{School of Mathematical Sciences\\ University of Southampton\\ United Kingdom}
\email{k.v.pavlov@soton.ac.uk}

\begin{document}

\begin{abstract}
  For an odd prime $p$, we show that the homotopy type of a closed
  $4$-manifold with fundamental group $\mathbb{Z}/p \rtimes \mathbb{Z}$ is determined by its
  quadratic $2$-type together with an additional $\mathbb{Z}/p$-valued
  intersection form. We also give a generalization to infinite groups
  of a result of Hambleton--Kreck on the polarized homotopy
  classification of Poincar\'e $4$-complexes, and we prove some
  partial results on the homotopy classification that are applicable
  more generally.
\end{abstract}

\maketitle

\section{Introduction}

The \emph{quadratic $2$-type} $Q(M)$ of a closed $4$-manifold $M$ is
the following collection of invariants
\[
  Q(M) := \left( \pi_1(M), \pi_2(M), k_M, w_M, \lambda_M \right),
\]
where $\pi_2(M)$ is considered as a $\mathbb{Z}[\pi_1(M)]$-module, $k_M \in
H^3(\pi_1(M), \pi_2(M))$ is the $k$-invariant classifying the Postnikov
$2$-type of $M$, $w_M\colon \pi_1(M) \to \{\pm1\}$ is the orientation
character, and $\lambda_M\colon \pi_2(M) \times \pi_2(M) \to \mathbb{Z}[\pi_1(M)]$ is the
equivariant intersection form. An \emph{isomorphism} of quadratic
$2$-types is a pair of isomorphisms on the fundamental group and the
second homotopy module which respect the $k$-invariant and induce
isomorphisms on the orientation character and the equivariant
intersection form.

These invariants play a large role in the study of non-simply
connected $4$-manifolds up to homotopy equivalence, and have been
shown to determine the homotopy type of the manifold in a wide variety
of cases depending on the fundamental group. For different classes of
finite groups, this was shown by Hambleton and Kreck~\cite{HK},
Kasprowski--Powell--Ruppik~\cite{KPRup}, and
Kasprowski--Nicholson--Ruppik~\cite{KNR}, all in the oriented
setting. Most recently Hillman--Kasprowski--Powell--Ray~\cite{KPR}
extended these results to some infinite groups in the general
setting. In particular, they investigated whether the quadratic
$2$-type determines the homotopy type of closed $4$-manifolds with
$3$-manifold fundamental groups, i.e.\ groups which arise as the
fundamental group of some (not necessarily orientable) closed
$3$-manifold. They obtained a positive answer under some conditions,
including the group $\mathbb{Z}/2 \times \mathbb{Z}$, the fundamental group of $\mathbb{RP}^2
\times S^1$, in the oriented setting. However, in general the groups $\mathbb{Z}/p \times
\mathbb{Z}$ for prime $p$ provide known examples of the failure of the
quadratic $2$-type to be a complete homotopy invariant.

\begin{example}
  \label{lens}
  Let $L$ and $L'$ be two $3$-dimensional lens spaces with isomorphic
  fundamental groups but distinct homotopy types (e.g.\ $L_{5, 1}$
  and $L_{5, 2}$). Then $L \times S^1$ and $L' \times S^1$ both have fundamental
  group $\mathbb{Z}/p \times \mathbb{Z}$ and also have distinct homotopy types, but $\pi_2(L \times
  S^1) = \pi_2(L' \times S^1) = 0$, so their quadratic $2$-types are
  trivially isomorphic.

  There is an analogous nonorientable example for Poincar\'e
  $4$-complexes. Let $L$ for $L'$ be as above, and let $q$ be a
  positive integer less than the order $p$ of their fundamental group
  such that $q^2 \equiv -1 \mod p$. Then by~\cite[Theorem~V]{olum} there
  are degree-$(-1)$ self maps on $L$ and $L'$ which induce
  multiplication by $q$ on the fundamental group. These maps must be
  self homotopy equivalences by Poincar\'e duality, and so their
  mapping tori $S^1\wt\times L$ and $S^1\wt\times L'$ are both nonorientable
  Poincar\'e $4$-complexes with fundamental group $\mathbb{Z}/p\rtimes\mathbb{Z}$ with the
  action of $\mathbb{Z}$ on $\mathbb{Z}/p$ given by multiplication by $q$. As in the
  previous example, they have distinct homotopy types but trivially
  isomorphic quadratic $2$-types.
\end{example}

In this article, we complete the homotopy classification for the
fundamental groups $\mathbb{Z}/p \rtimes \mathbb{Z}$ by including, in addition to the
invariants in the quadratic $2$-type, an additional $\mathbb{Z}/p$-valued
intersection form $\lambda_M^{\mathbb{Z}/p}$ on cohomology with local
coefficients. Although we state our results for manifolds in this
introduction, they hold more generally for finite Poincar\'e
$4$-complexes, and we will work in this generality during the rest of
the paper. All manifolds will be topological, and are assumed to be
closed and connected; Poincar\'e complexes are assumed to be finite
and connected.

We now define the form $\lambda_M^{\mathbb{Z}/p}$. Note that this will be a form on
cohomology with local coefficients, and so is only defined for
$4$-manifolds with fundamental group $\mathbb{Z}/p\rtimes\mathbb{Z}$. For an odd prime $p$ let
$\pi$ be the group $\mathbb{Z}/p\rtimes\mathbb{Z}$ with $\mathbb{Z}$ acting on $\mathbb{Z}/p$ by multiplication by
$q$, and consider the $\mathbb{Z}\pi$-module $\mathbb{Z}/p$ with the generator of the $\mathbb{Z}$
subgroup acting by multiplication by $q$ and the $\mathbb{Z}/p$ subgroup acting
trivially. Denote this module by $(\mathbb{Z}/p)^q$. The form $\lambda_M^{\mathbb{Z}/p}$ is
then defined on cohomology with local coefficients $H^2(M; (\mathbb{Z}/p)^q)$
via the cap product
\begin{align*}
  \lambda_M^{\mathbb{Z}/p} \colon H^2(M; (\mathbb{Z}/p)^q)\times H^2(M; (\mathbb{Z}/p)^q) & \to \mathbb{Z}/p \\ 
  (\alpha,\beta)                                          & \mapsto \beta\cap(\alpha\cap[M]).
\end{align*}

Having these forms be abstractly isomorphic will not be enough to give
us any information on the homotopy classification, so we will need a
notion of isomorphism compatible with a given isomorphism of the
quadratic $2$-type. We spell this out, along with a more detailed
definition of the form itself, in \cref{form}.

With this notion of isomorphism, the form $\lambda_M^{\mathbb{Z}/p}$ is able to
distinguish the homotopy inequivalent examples above, as we show
later, but what is more is that it combines with the quadratic
$2$-type to give a complete classification for these fundamental
groups. This is our main theorem.

\begin{theorem}
  \label{mainth}
  Let $p$ be an odd prime, and let $M$ and $M'$ be closed
  $4$-manifolds with fundamental group $\mathbb{Z}/p\rtimes\mathbb{Z}$. Then $M$ is homotopy
  equivalent to $M'$ if and only if they have isomorphic quadratic
  $2$-types and $\lambda_M^{\mathbb{Z}/p} \cong \lambda_{M'}^{\mathbb{Z}/p}$.
\end{theorem}

In~\cite{KPR}, the authors give criteria for deciding whether
the quadratic $2$-type determines the homotopy type. To prove
\cref{mainth}, we utilize their methods to obtain that, for $\mathbb{Z}/p\rtimes\mathbb{Z}$,
there are a finite number of homotopy types with a fixed quadratic
$2$-type, and then we show that the form $\lambda_M^{\mathbb{Z}/p}$ distinguishes
between these homotopy types.

The HKPR criteria are based on the following commutative diagram,
where $B$ is a $3$-coconnected space with fundamental group $\pi$ and $w
\colon \pi \to \{\pm1\}$ is a homomorphism; see \cref{comm-diag} for the
definition of the maps involved:
\begin{equation*}
  \begin{tikzcd}
    \mathbb{Z}^w \otimes_{\mathbb{Z}\pi} H_4(B; \mathbb{Z}\pi) \arrow[r, "\varphi_B"] \arrow[d, "\mathcal{B}_{H_2(B;, \mathbb{Z}\pi)}"]
    & H_4(B; \mathbb{Z}^w) \arrow[d, "\Theta_B"] \\
    \Her^w(H_2(B; \mathbb{Z}\pi)^{*}) \arrow[r, "\ev^{*}"]
    & \Her^w(H^2(B; \mathbb{Z}\pi)).
  \end{tikzcd}
\end{equation*}

Let $M$ and $M'$ be $4$-manifolds with fundamental group $\pi$,
orientation character $w$ and $3$-connected maps $f \colon M \to B$ and
$f' \colon M' \to B$ to a common Postnikov $2$-type. The two main
HKPR criteria are the following:
\begin{enumerate}
  \item\label{diff} The difference $f_{*}[M] - f_{*}'[M'] \in H_4(B;
    \mathbb{Z}^w)$ is in the image of the map $\varphi_B$.
  \item\label{kers} The kernel of the composition $\ev^{*}\circ\mathcal{B}_{H_2(B;
      \mathbb{Z}\pi)}$ is contained in the kernel of $\varphi_B$.
\end{enumerate}

We will see in \cref{sec1} that it is condition \eqref{diff} that
fails in \cref{lens}; in \cref{sec2} we will see that condition
\eqref{kers} can also fail for the fundamental groups $\mathbb{Z}/p\rtimes\mathbb{Z}$.

Along the way to \cref{mainth}, we will prove a few results which work
in greater generality than just the groups $\mathbb{Z}/p\rtimes\mathbb{Z}$. In \cref{sec0} we
will prove the following result on the classification of Poincar\'e
$4$-complexes with isomorphic quadratic $2$-types; it is a
generalization of Hambleton and Kreck's result for finite
fundamental groups in the oriented setting \cite[Theorem~1.1]{HK}.

\begin{theorem}
  \label{realization-intr}
  Let $M$ be a Poincar\'e $4$-complex with fundamental group $\pi$,
  orientation character $w$, and map to the Postnikov $2$-type $f
  \colon M \to B$. For every element $\beta \in \ker(\ev^{*}\circ\mathcal{B}_{H_2(B; \mathbb{Z}\pi)}) \subseteq
  \mathbb{Z}^w\otimes_{\mathbb{Z}\pi}H_4(B; \mathbb{Z}\pi)$, there exists a Poincar\'e $4$-complex $M_{\beta}$
  with the same quadratic $2$-type as $M$ and a $3$-connected map
  $f_{\beta} \colon M_{\beta} \to B$ such that $(f_{\beta})_{*}[M_{\beta}] = f_{*}[M] +
  \varphi_B(\beta) \in H_4(B; \mathbb{Z}^w)$.
\end{theorem}

As a corollary to this, we obtain the classification of polarized
homotopy types of Poincar\'e $4$-complexes, in the following sense.
Two Poincar\'e $4$-complexes $M$ and $M'$ with isomorphic quadratic
$2$-types are said to be \emph{homotopy equivalent over $B$} if
there is a homotopy equivalence $g \colon M \to M'$ such that $f'\circ g =
f$, where $f \colon M \to B$ and $f' \colon M \to B$ are the
maps to the (common) Postnikov $2$-type $B$.

\begin{corollary}
  \label{realization-cor}
  Let $M$ be a Poincar\'e $4$-complex with fundamental group $\pi$ and
  orientation character $w$, and let $B$ be the Postnikov $2$-type of
  $M$. Assume that, for every Poincar\'e $4$-complex $M'$ with the
  same quadratic $2$-type as $M$, condition \eqref{diff} of HKPR
  holds. Then there is a bijection 
  \[
    \mathcal{S}_4^{PD}(Q(M)) \cong \ker(\mathord{\ev^{*}} \circ \mathcal{B}_{H_2(B; \mathbb{Z}\pi)})/\ker\varphi_B,
  \]
  where $\mathcal{S}_4^{PD}(Q(M))$ is the set of homotopy types over $B$ of
  Poincar\'e $4$-complexes with the same quadratic $2$-type as $M$.
\end{corollary}

In \cref{sec-ker,sec-diff} we will prove some results which aid the
verification of the HKPR criteria for groups with $H^2(\pi; \mathbb{Z}\pi) = H^3(\pi;
\mathbb{Z}\pi) = 0$ under the assumption that the image $c_{*}[M] \in H_4(\pi; \mathbb{Z}^w)$
under the classifying map $M \to B\pi$ vanishes. \cref{sec-ker} contains
results which simplify condition \eqref{kers}, while in
\cref{sec-diff} we prove that condition \eqref{diff} is actually
satisfied, provided that $\pi$ is $1$-ended or that the homology with
certain coefficients $H_2(\pi; A^w)$ is trivial or cyclic of prime
order.

The rest of the paper is devoted to proving \cref{mainth}, with
\cref{sec-prop} concerned with some basic properties of the groups
$\mathbb{Z}/p\rtimes\mathbb{Z}$ and \cref{sec1,sec2} proving the theorem for stably-free and
non-stably-free $\pi_2(M)$ respectively.

\subsection*{Conventions and notation}
All Poincar\'e complexes are assumed to be finite.

For $a,b \in A$ we will write $a\sq b$ for the element $a\otimes b +
b\otimes a \in A\otimes A$ for the sake of brevity.

All $\mathbb{Z}\pi$-modules will be assumed to be finitely generated. They will
be left modules by default and unless specifically stated
otherwise. Homology and cohomology of groups and of spaces with local
coefficients takes coefficients in left modules. That is, homology is
formed by considering a free resolution of right modules and tensoring
with the coefficient module. Cohomology is formed by considering a
resolution of left modules.

Whenever a left module appears in a place where the context requires a
right module, this is accomplished using the involution on $\mathbb{Z}\pi$
sending $g\mapsto w(g)g^{-1}$. This allows for example forming the tensor
product over $\mathbb{Z}\pi$ of two left $\mathbb{Z}\pi$-modules. Similarly dual modules
will always be considered as left modules using this procedure.

The only exception to this rule are the symbols $\mathbb{Z}$ and $\mathbb{Z}^w$---the
former will always denote the trivial left or right module, while the
latter will always denote the $w$-twisted left or right module. This
results in some artifacts, e.g.\ when turning a short exact sequence
of left modules $A \to B \to \mathbb{Z}$ into a short exact sequence of right
modules, we obtain $A \to B \to \mathbb{Z}^w$.

The $\mathbb{Z}$-tensor product of left $\mathbb{Z}\pi$-modules $A\otimes B$ is a left
$\mathbb{Z}\pi$-module with the diagonal action. As per the conventions described
above, we can also form the $\mathbb{Z}\pi$-tensor product $A\otimes_{\mathbb{Z}\pi}B$, and we
have the isomorphism $A\otimes_{\mathbb{Z}\pi}B \cong \mathbb{Z}^w\otimes_{\mathbb{Z}\pi}(A\otimes B)$, where the twisting
of $\mathbb{Z}$ is required to account for the twisting in the involution when
turning $A$ into a right module.

\subsection*{Acknowledgements}
The author would like to thank Daniel Kasprowski for numerous
invaluable discussions and Maximilian Hans for providing helpful
feedback on the manuscript.

\section{Preliminaries}

In this section we lay out some preliminaries. In \cref{gamma} we
record some important information about Whitehead's universal
quadratic functor; in \cref{method} we explain the method of homotopy
classification we employ, viz.\ the one based on the criterion of
HKPR~\cite{KPR} extending the work of Hambleton and Kreck~\cite{HK};
and in \cref{form} we define the form $\lambda_M^{\mathbb{Z}/p}$.

\subsection{Whitehead's $\Gamma$ functor}
\label{gamma}

In~\cite{whitehead}, Whitehead defined the \emph{universal quadratic
  functor} on abelian groups. For torsion-free groups, $\Gamma(A)$ has the
description as the subgroup of $A\otimes A$ generated by $\{\,a \otimes a \mid a \in
A\,\}$ \cite[Sections~2.2 and 2.3]{baues-pirashvili}. If $A$ is free
abelian and $\mathfrak{B}$ is a basis, then $\Gamma(A)$ is free abelian with basis
$\{\,b \otimes b,\,b \otimes b' + b' \otimes b \mid b, b' \in \mathfrak{B}\,\}$. If $A$ is an
$R$-module, then $\Gamma(A)$ inherits an $R$-module structure from the
diagonal action on $A\otimes A$. If $B$ is a $3$-coconnected space with
fundamental group $\pi$, then we have that $\Gamma(H_2(B; \mathbb{Z}\pi)) \cong H_4(B; \mathbb{Z}\pi)$
as a direct consequence of~\cite[Sections~10 and 13]{whitehead}, and
we will often implicitly use this identification.

Hambleton and Kreck \cite{HK} investigated the $\mathbb{Z}\pi$-module structure
of $\Gamma(\mathbb{Z}\pi)$ for a group $\pi$, and this will be useful to us later:

\begin{lemma}[{\cite[Lemma~2.2]{HK}}]
  \label{Gamma-zpi-free}
  Let $\pi$ be a group with no $2$-torsion. Then $\Gamma(\mathbb{Z}\pi)$ is a free
  $\mathbb{Z}\pi$-module with basis given by $\{1 \otimes 1, 1 \otimes g + g \otimes 1 \mid g \in G \}$,
  where $G$ is the set containing exactly one of $g$ or $g^{-1}$ for
  every nonidentity $g \in \pi$.
\end{lemma}

The following lemma will help us use the form $\lambda_M^{\mathbb{Z}/p}$ to prove
\cref{mainth}.

\begin{lemma}
  \label{whitehead-eval}
  Let $B$ be a $3$-coconnected space with $\pi_2(B)$ free as an abelian
  group. Let $\alpha,\beta$ be elements of $H^2(\wt B; \mathbb{Z}/p) \cong \Hom_{\mathbb{Z}}(\pi_2(B),
  \mathbb{Z}/p)$. Then the map $\beta\cap(\alpha\cap-) \colon H_4(\wt B; \mathbb{Z}) \cong \Gamma(\pi_2(B)) \to \mathbb{Z}/p$
  is given by $a\otimes a \mapsto \alpha(a)\beta(a)$.
\end{lemma}

\begin{proof}
  The corresponding statement for $\mathbb{Z}$ coefficients can be found in
  \cite[page~96]{whitehead}, and we can deduce the lemma from this
  using that the cap product is natural in the coefficients as long as
  the reduction of coefficients map $H^2(\wt B; \mathbb{Z}) \to
  H^2(\wt B; \mathbb{Z}/p)$ is surjective. To see this, first note that
  $H_3(\wt B; \mathbb{Z}) = 0$ by the Hurewicz theorem. The map in question
  then has cokernel in $\Ext_{\mathbb{Z}}^1(H_2(\wt B; \mathbb{Z}), \mathbb{Z})$ by the universal
  coefficient theorem, which vanishes as $H_2(\wt B; \mathbb{Z}) \cong \pi_2(B)$ is
  free by assumption.
\end{proof}

Let $\pi$ be a group and $w\colon \pi \to \{\pm1\}$ be a homomorphism. For a
$\mathbb{Z}\pi$-module $A$ which is torsion-free as an abelian group, there is a
homomorphism
\[
  \mathcal{B}_A \colon \mathbb{Z}^w \otimes_{\mathbb{Z}\pi}\Gamma(A) \to \Her^w(A^{*})
\]
induced by $a \otimes b \mapsto ((f, g) \mapsto \ol{f(a)}g(b))$ as
in~\cite[Section~7]{hillman}. Here and in the sequel we use $A^{*}$ to
denote $\Hom_{\mathbb{Z}\pi}(A, \mathbb{Z}\pi)$ turned into a left $\mathbb{Z}\pi$-module using the
involution on $\mathbb{Z}\pi$ given by $g \mapsto w(g)g^{-1}$ for $g \in \pi$, and we use
$\Her^w(A)$ to denote the group of Hermitian forms on a module $A$
with the same involution.

And because it will be useful to us later, we give the following
lemma, a consequence of a theorem of Hillman:

\begin{lemma}[{\cite[Theorem~1]{hillman-2}}]
  \label{B-inj}
  Let $\pi$ be a group and $w \colon \pi \to \{\pm1\}$ be a homomorphism such
  that $w(g) = 1$ for all $g \in \pi$ of order $2$. For a finitely
  generated projective $\mathbb{Z}\pi$-module $A$, the map $\mathcal{B}_A$ is injective.
\end{lemma}

\begin{remark}
  \label{tors-free}
  This implies that $\mathbb{Z}^w \otimes_{\mathbb{Z}\pi}\Gamma(A)$ is torsion free, since
  $\Her^w(A^{*})$ is always torsion free.
\end{remark}

We end this section with the following, which describes the behavior
of the $\Gamma$ functor with respect to short exact sequences and thus
helps with calculations involving $\Gamma$ groups.

\begin{lemma}[cf.\ {\cite[Lemma~4]{bauer}}]
  \label{exact-seqs}
  Let $\pi$ be a group and suppose $0 \to A \to B \to C \to 0$ is a short exact
  sequence of $\mathbb{Z}\pi$-modules which are torsion-free as abelian
  groups. The induced maps on $\Gamma$ groups fit into a commutative
  diagram of $\mathbb{Z}\pi$-modules with exact rows and columns
  \begin{equation*}
    \begin{tikzcd}[column sep=small, row sep=scriptsize]
      & 0 \arrow[d]
      & 0 \arrow[d]
      & 
      & \\
      & \Gamma(A) \arrow[r, equals] \arrow[d]
      & \Gamma(A) \arrow[d]
      & 
      & \\
      0 \arrow[r]
      & K \arrow[r] \arrow[d, "f"]
      & \Gamma(B) \arrow[r] \arrow[d]
      & \Gamma(C) \arrow[r] \arrow[d, equals]
      & 0 \\
      0 \arrow[r]
      & A\otimes_{\mathbb{Z}}C \arrow[r, "g"] \arrow[d]
      & D \arrow[r] \arrow[d]
      & \Gamma(C) \arrow[r]
      & 0. \\
      &0&0&
      &
    \end{tikzcd}
  \end{equation*}

  If $A,B,C$ are free as abelian groups and $\{a_i\}$ and $\{a_i,
  c_j\}$ are bases for $A$ and $B$ respectively, then $\{a_i \otimes a_i, a_i
  \sq a_k, a_i \sq c_j\}$ is a basis for $K$ and the map $f$ is defined by
  sending $a_i \sq c_j \mapsto a_i  \otimes c_j$ and all other basis elements to
  $0$, while the map $g$ is defined by sending $a_i \otimes c_j \mapsto a_i \sq c_j
  \in D = \Gamma(B)/\Gamma(A)$.
\end{lemma}

\begin{proof}
  When $A,B,C$ are free abelian, the lemma is easily proved by
  considering bases. The general case follows from this since short
  exact sequences, tensor products, and the $\Gamma$ functor all commute
  with filtered colimits \cite[Section~9]{whitehead} and a
  torsion-free abelian group is the filtered colimit of its finitely
  generated free abelian subgroups.
\end{proof}

\begin{remark}
  \label{dsum-rem}
  If the exact sequence $0 \to A \to B \to C \to 0$ splits, then so do the
  $\Gamma$-functor sequences, and we obtain $\Gamma(A \oplus C) \cong \Gamma(A) \oplus \Gamma(C) \oplus A
  \otimes_{\mathbb{Z}} C$. We also have that the map $\mathcal{B}_{A \oplus C}$ splits as a direct
  sum as follows:
  \begin{equation*}
    \begin{tikzcd}[column sep=small]
      \mathbb{Z}^w \otimes_{\mathbb{Z}\pi} \Gamma(A) \arrow[r, phantom, "\oplus"] \arrow[d, "\mathcal{B}_A"]
      & \mathbb{Z}^w \otimes_{\mathbb{Z}\pi} \Gamma(C) \arrow[r, phantom, "\oplus"] \arrow[d, "\mathcal{B}_C"]
      & A \otimes_{\mathbb{Z}\pi} C \arrow[d, "a \otimes c \mapsto (f \mapsto (g \mapsto
      \ol{g(a)}f(c)))"] \\
      \Her^w(A^{*}) \arrow[r, phantom, "\oplus"]
      & \Her^w(C^{*}) \arrow[r, phantom, "\oplus"]
      & \Hom_{\mathbb{Z}\pi}(C^{*}, A^{**}),
    \end{tikzcd}
  \end{equation*}
  since $\Her^w((A \oplus C)^{*}) \cong \Her^w(A^{*}) \oplus \Her^w(C^{*}) \oplus
  \Hom_{\mathbb{Z}\pi}(C^{*}, A^{**})$. Here we have used that $\mathbb{Z}^w\otimes_{\mathbb{Z}\pi}(A\otimes C) \cong
  A\otimes_{\mathbb{Z}\pi}C$, where the tensor product $A \otimes_{\mathbb{Z}\pi} C$ is formed by
  considering $A$ as a right module using the involution on $\mathbb{Z}\pi$; see
  our conventions.
\end{remark}

\subsection{Method of homotopy classification}
\label{method}
The method of homotopy classification we employ is based on the
following result of Baues and Bleile extending \cite[Theorem~1.1]{HK}:

\begin{theorem}[{\cite[Corollary~3.2]{BB}}]
  \label{bauesbleile}
  Let $M$ and $N$ be Poincar\'e $4$-complexes and $B$ be a
  $3$-coconnected CW complex with $3$-connected maps $f \colon M \to B$
  and $g \colon N \to B$, and assume $M$ and $N$ have the same
  orientation character $w$. If $f_{*}([M]) = g_{*}([N]) \in H_4(B;
  \mathbb{Z}^w)$, then $M$ and $N$ are homotopy equivalent.
\end{theorem}

HKPR's criteria assures, when satisfied, that the equivariant
intersection form determines the image of the fundamental class in the
homology of the Postnikov $2$-type, whence classification via the
quadratic $2$-type follows by \cref{bauesbleile}. The HKPR criteria
are based on the following.

\begin{lemma}[{\cite[Lemma~10]{hillman}}]
  \label{comm-diag}
  Let $B$ be a $3$-coconnected space with fundamental group $\pi$ and
  $H_2(B; \mathbb{Z}\pi)$ torsion-free. Then there exists a commutative diagram
  \begin{equation*}
    \begin{tikzcd}
      \mathbb{Z}^w \otimes_{\mathbb{Z}\pi} H_4(B; \mathbb{Z}\pi) \arrow[r, "\varphi_B"] \arrow[d, "\mathcal{B}_{H_2(B;, \mathbb{Z}\pi)}"]
      & H_4(B; \mathbb{Z}^w) \arrow[d, "\Theta_B"] \\
      \Her^w(H_2(B; \mathbb{Z}\pi)^{*}) \arrow[r, "\ev^{*}"]
      & \Her^w(H^2(B; \mathbb{Z}\pi)),
    \end{tikzcd}
  \end{equation*}
  where the maps involved are the following:
  \begin{itemize}
  \item The map $\varphi_B$ is given by reduction of coefficients.
  \item $\mathcal{B}_{H_2(B; \mathbb{Z}\pi)}$ is the map defined above and
    in~\cite[Section~7]{hillman}, where we have identified $\Gamma(H_2(B;
    \mathbb{Z}\pi)) \cong H_4(B; \mathbb{Z}\pi)$.
  \item $\Theta_B$ is given by $x\mapsto\left( (\alpha, \beta) \mapsto \left\langle \beta, \alpha \cap x \right\rangle
    \right)$.
  \item The map $\ev^{*}$ is induced by the evaluation map $\ev \colon
    H^2(B; \mathbb{Z}\pi) \to H_2(B; \mathbb{Z}\pi)^{*}$.
  \end{itemize}
\end{lemma}

\begin{remark}
  \label{pi-2-tors-free}
  Note that $\pi_2(M)$ of a Poincar\'e $4$-complex $M$ is always
  torsion-free (see e.g.\ \cite[Proposition~3.1]{gurjar-pokale}).
  This means we are always free to apply the above lemma to the
  Postnikov $2$-type of $M$. We may do this without comment.
\end{remark}

We end this subsection with the following lemma, which will allow us
to split the proof of \cref{mainth} into two cases according to the
stable isomorphism class of $\pi_2(M)$.

\begin{lemma}
  \label{stab-iso}
  Let $M$ and $M'$ be Poincar\'e $4$-complexes with fundamental group
  $\pi$ and orientation character $w$, along with classifying maps $c_M
  \colon M \to B\pi$ and $c_{M'} \colon M' \to B\pi$. Then $\pi_2(M)$ and
  $\pi_2(M')$ are stably isomorphic as $\mathbb{Z}\pi$-modules if $(c_M)_{*}[M] =
  (c_{M'})_{*}[M'] \in H_4(\pi; \mathbb{Z}^w)/\pm\Aut(\pi)$.
\end{lemma}

\begin{proof}
  This is proved in \cite[Section~5]{KPT-2} for smooth $4$-manifolds,
  but the machinery used is algebraic and goes through for Poincar\'e
  complexes with minor modifications, using the cellular chain complex
  in the place of the one coming from a smooth handle decomposition.
\end{proof}

\subsection{The form $\lambda_M^{\mathbb{Z}/p}$}
\label{form}
For this subsection, let $M$ be a Poincar\'e $4$-complex with
fundamental group $\mathbb{Z}/p\rtimes\mathbb{Z}$, where $p$ is an odd prime and $\mathbb{Z}$ acts on
$\mathbb{Z}/p$ by multiplication by $q$. For a $\mathbb{Z}\pi$-module $A$, denote by $A^q$
the $\mathbb{Z}\pi$-module where the action is modified so that the generator $t$
of the $\mathbb{Z}$ subgroup acts as $t\cdot a = qta$ for $a\in A$. In particular,
the symbol $(\mathbb{Z}/p)^q$ denotes the module $\mathbb{Z}/p$ with $t$ acting by
multiplication by $q$ and the other generator acting trivially.

We define the form
\begin{align*}
  \lambda_M^{\mathbb{Z}/p} \colon H^2(M; (\mathbb{Z}/p)^q)\times H^2(M; (\mathbb{Z}/p)^q) & \to H_0(M;
                                                      (\mathbb{Z}/p)^q\otimes(\mathbb{Z}/p)^{wq})
  \\ 
  (\alpha,\beta)                                          & \mapsto \beta\cap(\alpha\cap[M]).
\end{align*}

The tensor product $(\mathbb{Z}/p)^q\otimes(\mathbb{Z}/p)^{wq}$, viewed as a left $\mathbb{Z}\pi$-module
with the diagonal action, is isomorphic to the module $(\mathbb{Z}/p)^{wq^2}$
where $t$ acts by $w(t)q^2$. The group $H_0(M; (\mathbb{Z}/p)^{wq^2}) \cong
\mathbb{Z}\otimes_{\mathbb{Z}\pi}(\mathbb{Z}/p)^{wq^2}$ is then trivial unless $q^2\equiv w(t)\mod\mathbb{Z}/p$, in
which case it is isomoprhic to $\mathbb{Z}/p$.

We wish to define a notion of isomorphism of this form which is
compatible with a given isomorphism of quadratic $2$-types. To do
this, we use the Serre spectral sequence with local coefficients of
the fibration $\wt M \to M \to B\pi$. This has $E_2$ page
$H^k(\pi; H^l(\wt M; \mathbb{Z}/p)^q)$ and converges to $H^{k+l}(M;
(\mathbb{Z}/p)^q)$. From the $2$-diagonal we obtain an exact sequence
\[
  0 \to H^2(\pi; (\mathbb{Z}/p)^q) \to H^2(M; (\mathbb{Z}/p)^q) \to H^0(\pi; H^2(\wt M;
  \mathbb{Z}/p)^q). 
\]
Next we note that for $H^2(\wt M; \mathbb{Z}/p) \cong \Hom_{\mathbb{Z}}(\pi_2(M),
\mathbb{Z}/p)$, the $\pi$-action induced by deck transformations is the usual
right action on $\Hom_{\mathbb{Z}}(\pi_2(M), \mathbb{Z}/p)$. As per our conventions, the
involution turns this into a left module. Twisting this by $q$ and
taking invariants, we see that the term on the right in the sequence
above is isomorphic to $\Hom_{\mathbb{Z}\pi}(\pi_2(M), (\mathbb{Z}/p)^q)$.

For $M$ and $M'$ with isomorphic quadratic $2$-types, we say that
$\lambda_M^{\mathbb{Z}/p}$ and $\lambda_{M'}^{\mathbb{Z}/p}$ are isomorphic, and write $\lambda_M^{\mathbb{Z}/p} \cong
\lambda_{M'}^{\mathbb{Z}/p}$, if there is an isomorphism $H^2(M; (\mathbb{Z}/p)^q) \cong H^2(M';
(\mathbb{Z}/p)^q)$ which is compatible with the forms and which fits into a
commutative diagram
\begin{equation*}
  \begin{tikzcd}
    H^2(\pi_1(M); (\mathbb{Z}/p)^q) \arrow[r] \arrow[d, "\cong"]
    & H^2(M; (\mathbb{Z}/p)^q) \arrow[r] \arrow[d, "\varphi"]
    & \Hom_{\mathbb{Z}\pi}(\pi_2(M), (\mathbb{Z}/p)^q) \arrow[d, "\cong"] \\
    H^2(\pi_1(M'); (\mathbb{Z}/p)^q) \arrow[r]
    & H^2(M'; (\mathbb{Z}/p)^q) \arrow[r]
    & \Hom_{\mathbb{Z}\pi}(\pi_2(M'), (\mathbb{Z}/p)^q),
  \end{tikzcd}
\end{equation*}
where the rows are the exact sequences as above and the vertical
isomorphisms on the left and the right are those induced by the
isomorphism on quadratic $2$-types.

\section{Polarized classification of Poincar\'e $4$-complexes}
\label{sec0}

In this section we will prove \cref{realization-intr}; we recall the
statement for the reader's convenience. This is a generalization of
\cite[Theorem~1.1]{HK}; throughout, we use the symbol $X$ instead of
$M$ to mirror the notation of that paper.

\begingroup
\def\thetheorem{\ref*{realization-intr}}
\begin{theorem}
  Let $X$ be a Poincar\'e $4$-complex with fundamental group $\pi$,
  orientation character $w$, and map to the Postnikov $2$-type
  $f \colon X \to B$. Using the notation from \cref{comm-diag}, for
  every element $\beta \in \ker(\mathord{\ev^{*}} \circ \mathcal{B}_{H_2(B; \mathbb{Z}\pi)}) \subseteq \mathbb{Z}^w
  \otimes_{\mathbb{Z}\pi} H_4(B; \mathbb{Z}\pi)$ there exists a Poincar\'e $4$-complex $X_{\beta}$
  with the same quadratic $2$-type as $X$ and a $3$-connected map
  $f_{\beta} \colon X_{\beta} \to B$  such that $(f_{\beta})_{*}[X_{\beta}] = f_{*}[X] +
  \varphi_B(\beta) \in H_4(B; \mathbb{Z}^w)$.
\end{theorem}
\addtocounter{theorem}{-1}
\endgroup

\begin{proof}
  To construct the required $X_{\beta}$, we follow
  \cite[page~89]{HK}. Write $X = X^{(3)}\cup_{\alpha}D^4$. Since $\beta$ lives in
  $\mathbb{Z}^w \otimes_{\mathbb{Z}\pi} H_4(B; \mathbb{Z}\pi)$, we can take a preimage in $\Gamma(H_2(B; \mathbb{Z}\pi)) \cong
  \Gamma(\pi_2(X))$. However, $\pi_2(X^{(2)})$ surjects onto $\pi_2(X)$, hence
  also $\Gamma(\pi_2(X^{(2)}))$ surjects onto $\Gamma(\pi_2(X))$, and so we can take
  a preimage $\gamma \in \Gamma(\pi_2(X^{(2)})) \subseteq \pi_3(X^{(2)})$. Form $X_{\beta} =
  X^{(3)}\cup_{\gamma+\alpha}D^4$ and extend $f|_{X^{(3)}}$ to a $3$-connected map
  $f_{\beta} \colon X_{\beta} \to B$. By this construction, $f_{\beta*}[X_{\beta}] =
  f_{*}[X] + \varphi_B(\beta)$ holds, where $[X_{\beta}] \in H_4(X_{\beta}; \mathbb{Z}^w) \cong \mathbb{Z}$ is
  the generator represented by the $4$-cell. Now $X_{\beta}$ has the same
  intersection form as $X$ since in the diagram from \cref{comm-diag}
  \begin{equation*}
    \begin{tikzcd}
      \mathbb{Z}^w \otimes_{\mathbb{Z}\pi} H_4(B; \mathbb{Z}\pi) \arrow[r, "\varphi_B"] \arrow[d, "\mathcal{B}_{H_2(B;, \mathbb{Z}\pi)}"]
      & H_4(B; \mathbb{Z}^w) \arrow[d, "\Theta_B"] \\
      \Her^w(H_2(B; \mathbb{Z}\pi)^{*}) \arrow[r, "\ev^{*}"]
      & \Her^w(H^2(B; \mathbb{Z}\pi))
    \end{tikzcd}
  \end{equation*}
  $\beta$ maps to $0$ by the left leg, hence $0 = \Theta_B\varphi_B(\beta) =
  \Theta_B(f_{\beta*}[X_{\beta}] - f_{*}[X]) = f_{\beta*}\lambda_{X_{\beta}} -
  f_{*}\lambda_X$. Together with the $3$-connected map $f_{\beta}$, this means
  that $Q(X_{\beta}) \cong Q(X)$. All that remains is to show that $X_{\beta}$ is
  indeed a Poincar\'e $4$-complex. That is, we want the maps
  \[
    -\cap[X_{\beta}] \colon H^k(X_{\beta}; \mathbb{Z}\pi) \to H_{4-k}(X_{\beta}; \mathbb{Z}\pi)
  \]
  to be isomorphisms.

  We focus on infinite groups, since the case for finite groups is
  covered by \cite[Theorem~1.1]{HK}. Since the boundaries of the
  $4$-cells differ by a cycle which lives on the $2$-skeleton, we have
  that the cellular chain complex $C_{*}(\wt X_{\beta})$ is
  isomorphic to $C_{*}(\wt X)$. This means that the above
  groups are already abstractly isomorphic because Poincar\'e duality
  holds for $X$. For $k = 0,3$ this is enough, since the groups
  involved are trivial. For $k = 2$ we have an isomorphism since
  $\lambda_{X_{\beta}}$, being the same as $\lambda_X$, is unimodular. This leaves $k
  = 1,4$.

  For $k = 4$, we can see the required isomorphism by considering the
  commutative diagram
  \begin{equation*}
    \begin{tikzcd}
      H^4(X_{\beta}; \mathbb{Z}\pi) \arrow[d, "\cong"'] \arrow[r, "-\cap{[X_{\beta}]}"]
      & H_0(X_{\beta}; \mathbb{Z}\pi) \arrow[d, "\cong"] \\
      H^4(X_{\beta}; \mathbb{Z}^w) \arrow[r, "\cong"', "-\cap{[X_{\beta}]}"]
      & H_0(X_{\beta}; \mathbb{Z}),
    \end{tikzcd}
  \end{equation*}
  where the vertical maps are surjections from $\mathbb{Z}$ to $\mathbb{Z}$, hence
  isomorphisms, and the bottom map is an isomorphism because it is
  given by evaluation on $[X_{\beta}]$.

  For $k = 1$, we will compare capping with $[X_{\beta}]$ to capping with
  $[X]$ by looking at the space $Y = X\cup_{\gamma+\alpha}D^4$ with both
  $4$-cells. Let $i_1 \colon X \to Y$ and $i_2 \colon X_{\beta} \to Y$ be
  the inclusions. Again, since the $4$-cells differ by a cycle on the
  $2$-skeleton, the homology and cohomology of $Y$ up to dimension
  $3$ agree with those of $X$ and $X_{\beta}$, and naturality of the cap
  product gives that the map $-\cap i_{1*}[X] \colon H^1(Y; \mathbb{Z}\pi) \to
  H_3(Y; \mathbb{Z}\pi)$ is an isomorphism. Hence we are done if we show
  that $-\cap i_{2*}[X_{\beta}] = -\cap i_{1*}[X]$ and apply naturality again.

  Recall that $\gamma \in \pi_3(X^{(2)})$, so we can form $K =
  X^{(2)}\cup_{\gamma}D^4$, which is a $4$-complex with no $3$-cells. The
  inclusion of the $2$-skeleton into $Y$ extends to a map $g \colon K
  \to Y$ and we have that $g_{*}[K] = i_{2*}[X_{\beta}] - i_{1*}[X]$. Now
  once again we use naturality of the cap product to see that $-\cap
  g_{*}[K] = g_{*}(g^{*}(-)\cap[K]) = 0$ since $K$ has no
  $3$-cells. This completes the proof of the case $k = 1$ and the
  theorem.
\end{proof}

Combining this theorem with \cref{bauesbleile} proves
\cref{realization-cor}.

\section{The kernel of $\varphi_B$}
\label{sec-ker}

The setup for this section and \cref{sec-diff} is the following. Let
$M$ be a Poincar\'e $4$-complex with fundamental group $\pi$ and
orientation character $w$, and suppose we have that $c_{*}[M] = 0 \in
H_4(\pi; \mathbb{Z}^w)$ for a classifying map $c\colon M \to B\pi$. We make no
assumption on the fundamental group for now. Let $K$ be a $2$-complex
with $\pi_1(K) \cong \pi$. Let $(C_*, \partial_{*})$ be the cellular chain complex of
free left $\mathbb{Z}\pi$-modules of $\wt K$. Then we have that $\pi_2(M)$ is
stably isomorphic to $\ker \partial_2\oplus\coker \partial^2$; see
\cite[Section~5]{KPT-2}. We write $B$ for the Postnikov $2$-type of
$M$ and $A$ for the module $\coker \partial^2$. Note that $\ker
\partial_2$ is isomorphic to $A^{*}$ by \cite[Lemma~3.12]{KPT}.

We will use the following lemma often throughout the paper,
sometimes implicitly.

\begin{lemma}
  \label{dthreeiso}
  Let $B'$ be a $3$-coconnected CW complex with fundamental group $\pi$,
  $\pi_2(B')$ stably isomorphic to $A\oplus A^{*}$, and $k$-invariant mapping
  to the $k$-invariant of $K$ under the projection map $H^3(\pi;
  \pi_2(B')) \cong H^3(\pi; A\oplus A^{*}) \to H^3(\pi; A^{*})$. Then, for $p\geq3$, the
  differential
  \[
    d_3 \colon H_p(\pi; \mathbb{Z}^w) \to H_{p-3}(\pi; \pi_2(B')^w) = H_{p-3}(\pi;
    (A^{*})^w)\oplus H_{p-3}(\pi; A^w)
  \]
  in the Serre spectral sequence with local coefficients of the
  fibration $\wt{B'} \to B' \to B\pi$ is injective with image the first
  summand.
\end{lemma}

\begin{proof}
  Consider first the Serre spectral sequence of the fibration
  $\wt K \to K \to B\pi$ that has $E^2$ page $H_p(\pi; H_q(\wt
  K; \mathbb{Z})^w)$ and converges to $H_{p+q}(K; \mathbb{Z}^w)$. Since $K$ and $\wt K$
  are $2$-complexes, the only nonvanishing rows in the spectral
  sequence are the zeroth and the second, and since $H_k(K; \mathbb{Z}^w) = 0$
  for $k\geq3$, we must have that the differential $d_3 \colon H_p(\pi;
  \mathbb{Z}^w) \to H_{p-3}(\pi; (A^{*})^w)$ is an isomorphism for $p\geq3$.

  Since the map $K \to P_2(K)$ to the Postnikov $2$-type is
  $3$-connected, the same $d_3$ differentials are also isomorphisms in 
  the analogous Serre spectral sequence of $K(A^{*}, 2) \to P_2(K) \to
  B\pi$.

  Now since the $k$-invariants agree, we have a map $B' \to P_2(K)$
  inducing a map of fibrations
  \begin{equation*}
    \begin{tikzcd}[sep=small]
      \wt{B'} \arrow[r] \arrow[d]
      & B' \arrow[r] \arrow[d]
      & B\pi \arrow[d, equals] \\
      K(A^{*}, 2) \arrow[r]
      & P_2(K) \arrow[r]
      & B\pi,
    \end{tikzcd}
  \end{equation*}
  and by naturality of the spectral sequence we obtain the commutative
  diagram
  \begin{equation*}
    \begin{tikzcd}
      H_p(\pi; \mathbb{Z}^w) \arrow[d, "\cong"] \arrow[r, "d_3"]
      & H_{p-3}(\pi; (A^{*})^w)\oplus H_{p-3}(\pi; A^w) \arrow[d] \\
      H_p(\pi; \mathbb{Z}^w) \arrow[r, "d_3", "\cong"']
      & H_{p-3}(\pi; (A^{*})^w),
    \end{tikzcd}
  \end{equation*}
  where the right vertical map is the projection, and we conclude the
  lemma.
\end{proof}

Recall that $B$ stands for the Postnikov $2$-type of $M$, and note
that it satisfies the conditions of \cref{dthreeiso}.

\begin{lemma}
  \label{triv-k}
  Let $L$ be a $3$-coconnected CW complex with vanishing
  $k$-invariant and $w \colon \pi_1(L) \to \{\pm1\}$ be a homomorphism. Then
  the map $\varphi_L \colon \mathbb{Z}^w\otimes_{\mathbb{Z}\pi_1(L)}H_4(\wt L; \mathbb{Z}) \to H_4(L; \mathbb{Z}^w)$ is
  injective.
\end{lemma}

\begin{proof}
  Consider the induced map
  \[
    \varphi_L^{\mathbb{Q}/\mathbb{Z}} \colon \Hom_{\mathbb{Z}}(H_4(L; \mathbb{Z}^w), \mathbb{Q}/\mathbb{Z}) \to
    \Hom_{\mathbb{Z}}(\mathbb{Z}^w\otimes_{\mathbb{Z}\pi_1(L)}H_4(\wt L; \mathbb{Z}), \mathbb{Q}/\mathbb{Z}).
  \]
  Injectivity of $\varphi_L$ is equivalent to surjectivity of $\varphi_L^{\mathbb{Q}/\mathbb{Z}}$;
  see e.g.\ \cite[Lemma~3.2.5]{weibel}. Now since $\mathbb{Q}/\mathbb{Z}$ is an
  injective group, the universal coefficient theorem identifies the
  map $\varphi_L^{\mathbb{Q}/\mathbb{Z}}$ with the edge homomorphism $H^4(L; (\mathbb{Q}/\mathbb{Z})^w) \to
  \Hom_{\mathbb{Z}\pi_1(L)}(\mathbb{Z}^w, H^4(\wt L; \mathbb{Q}/\mathbb{Z}))$ in the cohomological Serre
  spectral sequence of the fibration $\wt L \xrightarrow{\iota} L \to
  B\pi_1(L)$. In other words, we need to show that the image of the
  induced map on cohomology $\iota^{*} \colon H^4(L; (\mathbb{Q}/\mathbb{Z})^w) \to H^4(\wt L;
  \mathbb{Q}/\mathbb{Z})$ is the subgroup of invariants.

  Since the $k$-invariant of $L$ vanishes, $L$ is a twisted
  Eilenberg--MacLane space, whose cohomology has been investigated in
  connection to cohomology operations with local coefficients. In
  particular, \cite[Theorem~4.7(b)]{siegel} gives us that the
  invariants are in the image of $\iota^{*}$ as long as the group $H^3(\wt
  L; \mathbb{Q}/\mathbb{Z})$ is trivial. Again since $\mathbb{Q}/\mathbb{Z}$ is injective and $H_3(\wt L;
  \mathbb{Z}) = 0$ by the Hurewicz theorem, the universal coefficient theorem
  grants us that condition, and we are done.
\end{proof}

For the remainder of this section, let $B^{\oplus}$ denote the Postnikov
$2$-type of the double $D(K)$ of $K$. This has fundamental group $\pi$
and $\pi_2(B^{\oplus}) = A\oplus A^{*}$, and the inclusion $K \to D(K)$ induces the
inclusion of the $A^{*}$ summand on $\pi_2$. Hence the $k$-invariant of
$B^{\oplus}$ is contained in $H^3(\pi; A^{*})\leq H^3(\pi; A^{*}\oplus A)$ and $B^{\oplus}$
also satisfies the assumptions of \cref{dthreeiso}. Further let $B^s$
be the Postnikov $2$-type of $B^{\oplus}\vee^k S^2$. The inclusion and
collapse maps $B^{\oplus} \to B^{\oplus}\vee^k S^2 \to B^{\oplus}$ induce maps on the
Postnikov $2$-types $B^{\oplus} \to B^s \to B^{\oplus}$, which in turn induce the
inclusion of $\mathbb{Z}^w\otimes_{\mathbb{Z}\pi}\Gamma(A\oplus A^{*})$ as a direct summand in
$\mathbb{Z}^w\otimes_{\mathbb{Z}\pi}\Gamma(A \oplus A^{*} \oplus \mathbb{Z}\pi^k)$ from \cref{dsum-rem}.

On the other hand, $B^s$ is homotopy equivalent to the Postnikov
$2$-type of $B\vee^jS^2$, and we also have inclusion and collapse maps
inducing $B \to B^s \to B$, and $\mathbb{Z}^w\otimes_{\mathbb{Z}\pi}\Gamma(\pi_2(B))$ is a direct summand
of $\mathbb{Z}^w\otimes_{\mathbb{Z}\pi}\Gamma(A\oplus A^{*}\oplus\mathbb{Z}\pi^k)$ again by \cref{dsum-rem}. Finally, we
have that $\mathbb{Z}^w\otimes_{\mathbb{Z}\pi}\Gamma(A\oplus A^{*})$ decomposes as the sum
$\mathbb{Z}^w\otimes_{\mathbb{Z}\pi}\Gamma(A)\oplus\mathbb{Z}^w\otimes_{\mathbb{Z}\pi}\Gamma(A^{*})\oplus A\otimes_{\mathbb{Z}\pi}A^{*}$.

\begin{lemma}
  \label{ker_phi1}
  The kernel of the map $\varphi_B \colon \mathbb{Z}^w\otimes_{\mathbb{Z}\pi}H_4(B; \mathbb{Z}\pi) \to H_4(B; \mathbb{Z}^w)$
  is contained in the preimage of the $A\otimes_{\mathbb{Z}\pi}A^{*}$ summand under the
  inclusion $\mathbb{Z}^w\otimes_{\mathbb{Z}\pi}\Gamma(\pi_2(B)) \subseteq \mathbb{Z}^w\otimes_{\mathbb{Z}\pi}\Gamma(A\oplus A^{*}\oplus\mathbb{Z}\pi^k)$.
\end{lemma}

\begin{proof}
  Consider the Serre spectral sequence of the fibration $\wt{B^{\oplus}} \to
  B^{\oplus} \to B\pi$ that has $E^2$ page $H_p(\pi; H_q(\wt{B^{\oplus}}; \mathbb{Z})^w)$ and
  converges to $H_{p+q}(B^{\oplus}; \mathbb{Z}^w)$. The map $\varphi_{B^{\oplus}}$ is an edge
  homomorphism in this spectral sequence. Since $B^{\oplus}$ is
  $3$-coconnected, by the Hurewicz theorem we have that
  $H_3(\wt{B^{\oplus}}; \mathbb{Z}) = 0 = H_1(\wt{B^{\oplus}}; \mathbb{Z})$, so all entries in the
  first and third rows of the $E^2$ page vanish. The differentials
  $d_3 \colon H_p(\pi; \mathbb{Z}^w) \to H_{p-3}(\pi; (A^{*}\oplus A)^w) = H_{p-3}(\pi;
  (A^{*})^w)\oplus H_{p-3}(\pi; A^w)$ are given by inclusion of the first
  summand by \cref{dthreeiso}. In particular, $E_{5, 0}^3 = H_5(\pi;
  \mathbb{Z}^w)$ is killed by $d_3$ and does not survive to the $E^5$ page, so
  the kernel of the map $\varphi_{B^{\oplus}}$ is isomorphic to the image of the
  differential $d_3 \colon H_3(\pi; (A^{*})^w)\oplus H_3(\pi; A^w) \to
  \mathbb{Z}^w\otimes_{\mathbb{Z}\pi}H_4(\wt{B^{\oplus}}; \mathbb{Z}) = \mathbb{Z}^w\otimes_{\mathbb{Z}\pi}\Gamma(A\oplus A^{*})$.

  The same analysis applies unchanged for $B^s$, and naturality of the
  spectral sequence yields that $\ker\varphi_{B^s} \cong \ker\varphi_{B^{\oplus}}$ and is
  contained in the $\mathbb{Z}^w\otimes_{\mathbb{Z}\pi}\Gamma(A\oplus A^{*})$ summand of $\mathbb{Z}^w\otimes_{\mathbb{Z}\pi}\Gamma(A\oplus
  A^{*}\oplus\mathbb{Z}\pi^k)$. But on the other hand the map $B \to B^s$ induces an
  injection $\ker\varphi_B \to \ker\varphi_{B^s}$. Hence to prove the lemma, we have
  to show that the kernel of $\varphi_{B^{\oplus}}$ is contained in the $A\otimes_{\mathbb{Z}\pi}
  A^{*}$ summand of $\mathbb{Z}^w\otimes_{\mathbb{Z}\pi}\Gamma(A \oplus A^{*})$.

  For this we will show that the kernel of $\varphi_{B^{\oplus}}$ projects onto
  the $\mathbb{Z}^w\otimes_{\mathbb{Z}\pi}\Gamma(A)$ and $\mathbb{Z}^w\otimes_{\mathbb{Z}\pi}\Gamma(A^{*})$ summands trivially. For
  the latter, let $P_2(K)$ denote the Postnikov $2$-type of $K$. As in
  the proof of \cref{dthreeiso}, since the $k$-invariants agree we
  have a map $B^{\oplus} \to P_2(K)$ inducing a map of fibrations. In the
  analagous spectral sequence for $K(A^{*}, 2) \to P_2(K) \to B\pi$, the
  differential $d_3 \colon H_3(\pi; (A^{*})^w) \to \mathbb{Z}^w\otimes_{\mathbb{Z}\pi}H_4(K(A^{*},
  2); \mathbb{Z}) = \mathbb{Z}^w\otimes_{\mathbb{Z}\pi}\Gamma(A^{*})$ is trivial, since the differential $d_3
  \colon H_6(\pi; \mathbb{Z}^w) \to H_3(\pi; (A^{*})^w)$ is again an isomorphism as
  in the proof of \cref{dthreeiso}. Naturality of the spectral
  sequence induces the commutative diagram
  \begin{equation*}
    \begin{tikzcd}
      H_3(\pi; (A \oplus A^{*})^w) \arrow[d] \arrow[r, "d_3"]
      & \mathbb{Z}^w\otimes_{\mathbb{Z}\pi}\Gamma(A\oplus A^{*}) \arrow[d] \\
      H_3(\pi; (A^{*})^w) \arrow[r, "d_3", "0"']
      & \mathbb{Z}^w\otimes_{\mathbb{Z}\pi}\Gamma(A^{*}),
    \end{tikzcd}
  \end{equation*}
  where the vertical maps are the projections. Hence the image of
  $d_3$, and therefore the kernel of $\varphi_{B^{\oplus}}$, projects onto the
  $\mathbb{Z}^w\otimes_{\mathbb{Z}\pi}\Gamma(A^{*})$ summand trivially.

  On the other hand, let $L$ denote the $3$-coconnected space with
  $\pi_1 \cong \pi$, $\pi_2 \cong A$ and trivial $k$-invariant. The $k$-invariant of
  $B^{\oplus}$ has trivial component in $H^3(\pi; A) \leq H^3(\pi; A\oplus A^{*})$, so
  we have a map $B^{\oplus} \to L$ inducing a commutative diagram
  \begin{equation*}
    \begin{tikzcd}
      \mathbb{Z}^w\otimes_{\mathbb{Z}\pi}\Gamma(A\oplus A^{*}) \arrow[d] \arrow[r, "\varphi_{B^{\oplus}}"]
      & H_4(B^{\oplus}; \mathbb{Z}^w) \arrow[d] \\
      \mathbb{Z}^w\otimes_{\mathbb{Z}\pi}\Gamma(A) \arrow[r, "\varphi_L"]
      & H_4(L; \mathbb{Z}^w),
    \end{tikzcd}
  \end{equation*}
  where the left vertical map is the projection. Now $\varphi_L$ is
  injective by \cref{triv-k}, hence the kernel of $\varphi_{B^{\oplus}}$ projects
  onto the $\mathbb{Z}^w\otimes_{\mathbb{Z}\pi}\Gamma(A)$ summand trivially, as required.
\end{proof}

Next we will prove the main theorem of this section, but for this we
need to introduce an assumption on the fundamental group. Recall that
a $\mathbb{Z}\pi$-module $L$ is called \emph{torsionless} if it includes into a
free module. Equivalently, $L$ is torsionless if the natural map
$\ev_L \colon L \to L^{**}$ sending $a\in L$ to $(f\mapsto\ol{f(a)})\in
L^{**}$ is injective, and $L$ is \emph{reflexive} if $\ev_L$ is an
isomorphism. In this section we will need the assumption that $A$ is
torsionless, while reflexivity will be assumed in \cref{sec-diff}.

\begin{remark}
  \label{assum-rem}
  These assumptions can be formulated in terms of the cohomology of
  $\pi$, since by the universal coefficients spectral sequence there is
  an exact sequence as follows (see \cite[Remark~4.2]{KPR}):
  \[
    0 \to H^2(\pi; \mathbb{Z}\pi) \to A \xrightarrow{\ev_A} A^{**} \to H^3(\pi; \mathbb{Z}\pi) \to 0.
  \]
  By this sequence $A$ being torsionless is equivalent to having
  $H^2(\pi; \mathbb{Z}\pi)$ vanish and reflexivity is equivalent to in addition
  having $H^3(\pi; \mathbb{Z}\pi)$ vanish. For example, both conditions are
  satisfied by virtual duality groups of dimension $1$ or at least
  $4$.
\end{remark}

Recall that $\mathbb{Z}^w\otimes_{\mathbb{Z}\pi}\Gamma(\pi_2(B))$ is a direct summand of $\mathbb{Z}^w\otimes_{\mathbb{Z}\pi}\Gamma(A\oplus
A^{*}\oplus\mathbb{Z}\pi^k)$, so by \cref{dsum-rem} the map $\mathcal{B}_{\pi_2(B)}$ is a direct
summand of $\mathcal{B}_{A\oplus A^{*}\oplus\mathbb{Z}\pi^k}$, which on the other hand also includes
the map $A^{*} \otimes_{\mathbb{Z}\pi} A \to \Hom_{\mathbb{Z}\pi}(A^{**}, A^{**})$ as a direct
summand.

\begin{theorem}
  \label{ker_phi2}
  Let $L$ be the preimage of the
  kernel of the map $A^{*} \otimes_{\mathbb{Z}\pi} A \to \Hom_{\mathbb{Z}\pi}(A^{**}, A^{**})$ under
  the inclusion $\mathbb{Z}^w\otimes_{\mathbb{Z}\pi}\Gamma(\pi_2(B)) \subseteq \mathbb{Z}^w\otimes_{\mathbb{Z}\pi}\Gamma(A\oplus A^{*}\oplus\mathbb{Z}\pi^k)$. Then
  $L$ is contained in the kernel of the map $\varphi_B \colon
  \mathbb{Z}^w\otimes_{\mathbb{Z}\pi}H_4(B; \mathbb{Z}\pi) \to H_4(B; \mathbb{Z}^w)$.
\end{theorem}

\begin{proof}
  We have the commutative diagram
  \begin{equation*}
    \begin{tikzcd}[row sep=small]
      \mathbb{Z}^w\otimes_{\mathbb{Z}\pi}\Gamma(A\oplus A^{*}) \arrow[r, "\varphi_{B^{\oplus}}"] \arrow[d]
      & H_4(B^{\oplus}; \mathbb{Z}^w) \arrow[d] \\
      \mathbb{Z}^w\otimes_{\mathbb{Z}\pi}\Gamma(A\oplus A^{*}\oplus\mathbb{Z}\pi^k) \arrow[r, "\varphi_{B^s}"] \arrow[d]
      & H_4(B^s; \mathbb{Z}^w) \arrow[d] \\
      \mathbb{Z}^w\otimes_{\mathbb{Z}\pi}\Gamma(\pi_2(B)) \arrow[r, "\varphi_B"]
      & H_4(B; \mathbb{Z}^w).
    \end{tikzcd}
  \end{equation*}
  Since $\mathbb{Z}^w\otimes_{\mathbb{Z}\pi}\Gamma(\pi_2(B))$ is a direct summand of
  $\mathbb{Z}^w\otimes_{\mathbb{Z}\pi}\Gamma(A\oplus A^{*}\oplus\mathbb{Z}\pi^k)$, we have that $L$ pulls back to the
  kernel of the map $A^{*} \otimes_{\mathbb{Z}\pi} A \to \Hom_{\mathbb{Z}\pi}(A^{**}, A^{**})$ on
  the top row of the diagram. Therefore it is enough to prove that
  this kernel is contained in the kernel of the map $\varphi_{B^{\oplus}}$.
  
  Now recall that $A^{*}$ is isomorphic to $\ker d_2$ in the chain
  complex $(C_{*}, d_{*})$ and consider the commutative diagram
  \begin{equation*}
    \begin{tikzcd}[column sep = small]
      A^{*}\otimes_{\mathbb{Z}\pi}A \arrow[r, "\theta"] \arrow[d]
      & C_2 \otimes_{\mathbb{Z}\pi} A \arrow[d] \\
      \Hom_{\mathbb{Z}\pi}(A^{**}, A^{**}) \arrow[r]
      & \Hom_{\mathbb{Z}\pi}(C_2^{*}, A^{**}).
    \end{tikzcd}
  \end{equation*}
  The right map can be identified with the map $A\otimes_{\mathbb{Z}\pi}C_2 \to
  \Hom_{\mathbb{Z}\pi}(A^{*}, C_2^{**})$, which on the other hand can be
  identified with a sum of copies of $\ev_A$. By our assumption that
  $A$ is torsionless, this is injective. Hence in the diagram, the
  kernel of the left map, which is the map we are interested in, is
  contained in the kernel of the top map $\theta$.

  By definition the group $H_2(K; A^w)$ is computed by tensoring the
  cellular chain complex of right modules of $\wt K$ with the
  left module $A^w$. This complex is the same as the left-module
  resolution $(C_{*}^w, d_{*}^w)$ of the left module $\mathbb{Z}^w$ turned into
  right modules via the involution on $\mathbb{Z}\pi$ as per our
  conventions. Hence we have that $H_2(K; A^w)$ is the kernel of the
  map $C_2^w\otimes_{\mathbb{Z}\pi}A^w \to C_1^w\otimes_{\mathbb{Z}\pi}A^w$. The twistings therefore
  cancel and we can identify $H_2(K; A^w)$ with the kernel of the map
  $C_2\otimes_{\mathbb{Z}\pi}A \to C_1\otimes_{\mathbb{Z}\pi}A$, which contains the image of
  $\theta$. Moreover, the Postnikov $2$-type $P_2(K)$ of $K$ also has
  $H_2(P_2(K); A^w)$ the kernel of the map $C_2^w\otimes_{\mathbb{Z}\pi}A^w \to
  C_1^w\otimes_{\mathbb{Z}\pi}A^w$, as $P_2(K)$ is obtained from $K$ by adding cells in
  dimension $4$ and higher. Therefore we write $\theta$ as a map
  \[
    \theta \colon A^{*}\otimes_{\mathbb{Z}\pi}A \to H_2(P_2(K); A^w).
  \]

  Note that the decomposition of $H_4(\wt {B^{\oplus}}; \mathbb{Z}) \cong \Gamma(A\oplus A^{*})$
  into $\Gamma(A)\oplus(A\otimes_{\mathbb{Z}} A^{*})\oplus\Gamma(A^{*})$ from \cref{dsum-rem} can be seen
  in the Serre spectral sequence of the fibration $K(A, 2) \to
  \wt {B^{\oplus}} \to K(A^{*}, 2)$, where these summands appear on the
  $4$-diagonal of the $E^2$ page as
  \begin{align*}
    H_0(K(A^{*}, 2); H_4(K(A, 2); \mathbb{Z})) & \cong \Gamma(A) \\
    H_2(K(A^{*}, 2); H_2(K(A, 2); \mathbb{Z})) & \cong A^{*}\otimes_{\mathbb{Z}} A \\
    H_4(K(A^{*}, 2); H_0(K(A, 2); \mathbb{Z})) & \cong \Gamma(A^{*}).
  \end{align*}

  Compare this to the Serre spectral sequence for the fibration $K(A,
  2) \to B^{\oplus} \to P_2(K)$, this time with coefficients twisted by $w$. The
  map between the $E_{2,2}^2$-terms is the natural map $A^{*}\otimes_{\mathbb{Z}} A \to
  H_2(P_2(K); A^w)$, which factors through $\theta$. The map
  $H_4(\wt {B^{\oplus}}; \mathbb{Z}) \to H_4(B^{\oplus};\mathbb{Z}^w)$ is also induced by the
  comparison, and factors through $\varphi_{B^{\oplus}}$. By naturality we have the
  following commutative diagram
  \begin{equation*}
    \begin{tikzcd}[column sep=tiny]
      \mathbb{Z}^w\otimes_{\mathbb{Z}\pi}H_4(\wt {B^{\oplus}}; \mathbb{Z}) \arrow[r, "\varphi_{B^{\oplus}}"]
      & H_4(B^{\oplus}; \mathbb{Z}^w) \\
      (\mathbb{Z}^w\otimes_{\mathbb{Z}\pi}\Gamma(A))\oplus(A^{*}\otimes_{\mathbb{Z}\pi}A) \arrow[r] \arrow[u, hook] \arrow[d]
      & F_{2,2} \arrow[u, hook] \arrow[d] \\
      A^{*}\otimes_{\mathbb{Z}\pi}A \arrow[r, "\theta"]
      & H_2(P_2(K); A^w)/\im d_3,
    \end{tikzcd}
  \end{equation*}
  where $F_{2,2}$ is a subgroup in the filtration computing
  $H_4(B^{\oplus}; \mathbb{Z}^w)$ and the lower left vertical map is the
  projection. An element in the kernel of $\theta$ lifts to some element in
  the kernel of the middle horizontal map. However, by \cref{ker_phi1}
  this is precisely the kernel of $\varphi_{B^{\oplus}}$ and the lift must have
  no component in the $\mathbb{Z}^w\otimes_{\mathbb{Z}\pi}\Gamma(A)$ summand. Thus $\ker\theta \subseteq
  \ker\varphi_{B^{\oplus}}$ as desired, and the proof is complete.
\end{proof}

\begin{corollary}
  \label{ker-phi-cor}
  Suppose $A$ is torsionless and the map $\ev^{*} \colon
  \Her^w(H_2(B; \mathbb{Z}\pi)^{*}) \to \Her^w(H^2(B; \mathbb{Z}\pi))$ is injective. Then the
  kernel of the map $\varphi_B \colon \mathbb{Z}^w\otimes_{\mathbb{Z}\pi}H_4(B; \mathbb{Z}\pi) \to H_4(B; \mathbb{Z}^w)$
  coincides with the preimage of the kernel of the map $A^{*} \otimes_{\mathbb{Z}\pi} A
  \to \Hom_{\mathbb{Z}\pi}(A^{**}, A^{**})$ under the inclusion of
  $\mathbb{Z}^w\otimes_{\mathbb{Z}\pi}\Gamma(\pi_2(B))$ into $\mathbb{Z}^w\otimes_{\mathbb{Z}\pi}\Gamma(A\oplus A^{*}\oplus\mathbb{Z}\pi^k)$.
\end{corollary}

\begin{proof}
  The assumption that the map $\ev^{*}$ is injective grants us that
  $\ker\varphi_B \subseteq \ker\mathcal{B}_{\pi_2(B)}$ by virtue of \cref{comm-diag}. Combining
  \cref{ker_phi1,ker_phi2} then proves the corollary.
\end{proof}

\begin{remark}
  The injectivity assumption on $\ev^{*}$ is implied by, but not
  equivalent to, the condition that $H^3(\pi; \mathbb{Z}\pi) = 0$. See
  \cite[Section~5]{KPR}.
\end{remark}

\section{The image of $\varphi_B$}
\label{sec-diff}

For this section we keep the assumption that $c_{*}[M] = 0 \in H_4(\pi;
\mathbb{Z}^w)$ and all of the notation from the previous section, and we show
that the kernel of the map
\begin{align*}
  \Theta_B \colon H_4(B; \mathbb{Z}^w) & \to \Her^w(H^2(B; \mathbb{Z}\pi)) \\
  x                     & \mapsto ((\alpha,\beta)\mapsto\langle\beta,\alpha\cap x\rangle)
\end{align*}
is contained in the image of the map $\varphi_B \colon \mathbb{Z}^w\otimes_{\mathbb{Z}\pi}H_4(B;
\mathbb{Z}\pi) \to H_4(B; \mathbb{Z}^w)$, subject to the assumptions that the module $A$ is
reflexive and that the group $H_2(\pi; A^w)$ is cyclic of prime order.

The proofs in this section are generalized from
\cite[Section~8]{KPR}.

\begin{lemma}[cf.\ {\cite[Lemma~8.11]{KPR}}]
  \label{nontriv-im}
  Suppose that $H_2(\pi; A^w)$ is finite cyclic. Then there is an exact
  sequence
  \[
    \mathbb{Z}^w \otimes_{\mathbb{Z}\pi} H_4(B; \mathbb{Z}\pi) \xrightarrow{\varphi_B} H_4(B; \mathbb{Z}^w) \to H_2(\pi;
    A^w) \to 0,
  \]
  where the image of $f_{*}[M]$ generates $H_2(\pi; A^w)$.
\end{lemma}

\begin{proof}
  Consider the Serre spectral sequence of the fibration $\wt B
  \to B \to B\pi$ that has $E^2$ page $H_p(\pi; H_q(\wt B; \mathbb{Z})^w)$ and
  converges to $H_{p+q}(B; \mathbb{Z}^w)$. Since $B$ is $3$-coconnected, by the
  Hurewicz Theorem $H_3(\wt B; \mathbb{Z})$ is trivial. Hence the
  differential $d_2 \colon H_2(\pi; (A^{*}\oplus A)^w) \to H_0(\pi;
  H_3(\wt B; \mathbb{Z})^w) \cong 0$ is trivial and the term $H_2(\pi; (A^{*}\oplus
  A)^w)$ survives to the $E^3$ page. By \cref{dthreeiso} the
  differentials $d_3 \colon H_p(\pi; \mathbb{Z}^w) \to H_{p-3}(\pi; (A^{*}\oplus A)^w)$
  are given by inclusion of the first summand. This means that on the
  $4$-diagonal of the $E^{\infty}$ page the only nontrivial terms,
  comprising the associated graded group for $H_4(B; \mathbb{Z}^w)$, are a
  quotient of $\mathbb{Z}^w \otimes_{\mathbb{Z}\pi} H_4(\wt B; \mathbb{Z})$ and $H_2(\pi; A^w)$.
  
  Compare now to the analogous spectral sequence of $\wt M \to M
  \to B\pi$. Note that $H_3(\wt M; \mathbb{Z}) \cong H_3(M; \mathbb{Z}\pi) \cong H^1(M; \mathbb{Z}\pi^w) \cong
  H^1(\pi; \mathbb{Z}\pi^w)$. We have that the differential $d_2 \colon H_2(\pi;
  (A^{*}\oplus A)^w) \to H_0(\pi; H_3(\wt M; \mathbb{Z})^w) \cong \mathbb{Z}^w\otimes_{\mathbb{Z}\pi}H^1(\pi;
  \mathbb{Z}\pi)^w$ has torsion-free codomain. But by assumption the $H_2(\pi; A^w)$
  summand of the domain is torsion, hence it survives to the
  $E^3$-page. On that page the analysis is the same as above by
  naturality, as the map $M \to B$ is $3$-conneceted, and we obtain the
  commutative diagram with exact rows
  \begin{equation*}
    \begin{tikzcd}
      & H_4(M; \mathbb{Z}^w) \arrow[r] \arrow[d, "f_{*}"]
      & H_2(\pi; A^w) \arrow[r] \arrow[d, "\cong"]
      & 0 \\
      \mathbb{Z}^w \otimes_{\mathbb{Z}\pi} H_4(\wt B; \mathbb{Z}) \arrow[r, "\varphi_B"]
      & H_4(B; \mathbb{Z}^w) \arrow[r]
      & H_2(\pi; A^w) \arrow[r]
      & 0,
    \end{tikzcd}
  \end{equation*}
  which shows that $f_{*}[M]$ maps to a generator of $H_2(\pi; A^w)$
  as required.
\end{proof}

\begin{lemma}[cf.\ {\cite[Lemma~8.3]{KPR}}]
  \label{stab-inv-right-cont-top}
  Let $B'$ be a $3$-coconnected CW-complex with fundamental group
  $\pi$. Write $B^s$ for the Postnikov $2$-type of the space $B' \vee
  S^2$. Then $\ker\Theta_{B'} \subseteq \im\varphi_{B'}$ if and only if $\ker\Theta_{B^s} \subseteq
  \im\varphi_{B^s}$.
\end{lemma}

\begin{proof}
  The inclusion and collapse maps induce maps $B' \to B^s \to B'$. These
  induce the following commutative diagram with exact rows, using the
  Serre spectral sequence as in the proof of \cref{nontriv-im}
  (note that we do not need the hypothesis of that lemma here):
  \begin{equation*}
    \begin{tikzcd}[row sep=small]
      \mathbb{Z}^w \otimes_{\mathbb{Z}\pi} H_4(B'; \mathbb{Z}\pi) \arrow[r, "\varphi_{B'}"] \arrow[d]
      & H_4(B'; \mathbb{Z}^w) \arrow[d] \arrow[r]
      & H_2(\pi; A^w) \arrow[r] \arrow[d, "\cong"]
      & 0 \\
      \mathbb{Z}^w \otimes_{\mathbb{Z}\pi} H_4(B^s; \mathbb{Z}\pi) \arrow[r, "\varphi_{B^s}"] \arrow[d]
      & H_4(B^s; \mathbb{Z}^w) \arrow[d] \arrow[r]
      & H_2(\pi; A^w) \arrow[r] \arrow[d, "\cong"]
      & 0 \\
      \mathbb{Z}^w \otimes_{\mathbb{Z}\pi} H_4(B'; \mathbb{Z}\pi) \arrow[r, "\varphi_{B'}"]
      & H_4(B'; \mathbb{Z}^w) \arrow[r]
      & H_2(\pi; A^w) \arrow[r]
      & 0.
    \end{tikzcd}
  \end{equation*}
  
  Suppose that $\ker\Theta_{B'}\subseteq\im\varphi_{B'}$, and take $x\in\ker\Theta_{B^s}$. The
  image of $x$ in $H_4(B'; \mathbb{Z}^w)$ is in the kernel of $\Theta_{B'}$, hence
  in the image of $\varphi_{B'}$ by assumption. Therefore $x$ maps trivially
  to $H_2(\pi; A^w)$, and so is in the image of $\varphi_{B^s}$ as
  required. The other direction is analogous.
\end{proof}

\begin{theorem}[cf.\ {\cite[Lemma~8.13]{KPR}}]
  \label{right-cont-top-general}
  Suppose that $A$ is reflexive and that $H_2(\pi; A^w)$ is trivial or
  cyclic of prime order. Then the kernel of the map $\Theta_B$ is contained
  in the image of the map $\varphi_B$.
\end{theorem}

\begin{proof}
  If $N$ is another Poincar\'e $4$-complex with $c_{*}[N] = 0$ and
  $B'$ is its Postnikov $2$-type, then we have that $P_2(B\vee^jS^2) \simeq
  P_2(B'\vee^kS^2)$ for some $j,k\in\mathbb{Z}$. But now by
  \cref{stab-inv-right-cont-top} it suffices to prove the statement of
  the theorem for a single concrete $N$.

  Hence suppose that $B$ is the Postnikov $2$-type of the double $N$
  of $K$. By \cite[Lemma~7.12]{KPT} we have that the element $\lambda
  \coloneqq f_{*}\lambda_N \in \Her^w(A\oplus A^{*})$ maps to $(0, *, \id)$ under
  the decomposition $\Her^w(A\oplus A^{*}) \cong \Her^w(A) \oplus \Her^w(A^{*}) \oplus
  \Hom_{\mathbb{Z}\pi}(A^{**}, A^{**})$.

  Now let $y \in \ker \Theta_B$. By \cref{comm-diag,nontriv-im} we get a
  commutative diagram
  \begin{equation*}
    \begin{tikzcd}
      \mathbb{Z}^w \otimes_{\mathbb{Z}\pi} H_4(B; \mathbb{Z}\pi) \arrow[r, "\varphi_B"] \arrow[d, "\mathcal{B}_{H_2(B;, \mathbb{Z}\pi)}"]
      & H_4(B; \mathbb{Z}^w) \arrow[d, "\Theta_B"] \arrow[r]
      & H_2(\pi; A^w) \arrow[r]
      & 0 \\
      \Her^w(H_2(B; \mathbb{Z}\pi)^{*}) \arrow[r, "\ev^{*}"]
      & \Her^w(H^2(B; \mathbb{Z}\pi)).
    \end{tikzcd}
  \end{equation*}
  The top row is exact and so, if $H_2(\pi; A^w)$ is trivial, the
  theorem follows. Otherwise, for some positive integer $k \leq |H_2(\pi;
  A^w)|$, we have that $kf_{*}[N] - y$ is in the image of $\varphi_B$. Under
  $\Theta_B$, this element maps to $k\lambda$. Since $A$ is reflexive, we know
  that $H^2(\pi; \mathbb{Z}\pi) = H^3(\pi; \mathbb{Z}\pi)=0$ (see \cref{assum-rem}), and so the
  map $\ev^{*}$ is an isomorphism by \cite[Lemma~5.1]{KPR}. Therefore,
  it suffices to show that $k\lambda \in \Her^w(H_2(B; \mathbb{Z}\pi)^{*})$ is contained
  in the image of $\mathcal{B}_{H_2(B; \mathbb{Z}\pi)}$ only if $k$ is equal to the order
  of $H_2(\pi; A^w)$. Then $y$ will map to $0$ in $H_2(\pi; A^w)$, and
  hence be in the image of $\varphi_B$ as required.

  Suppose $k\lambda$ is in the image of $\mathcal{B}_{H_2(B; \mathbb{Z}\pi)}$, i.e.\
  $k\id_{A^{**}}$ is in the image of the map $A\otimes_{\mathbb{Z}\pi} A^{*} \to
  \Hom_{\mathbb{Z}\pi}(A^{**}, A^{**}) \cong \Hom_{\mathbb{Z}\pi}(A, A)$, where the isomorphism
  uses that $A$ is reflexive. The composition is given by the natural
  map sending $a\otimes a' \mapsto (x \mapsto a'(x)a)$. According to Bourbaki
  \cite[Chapter~II, Section~4.2, Remark~1]{bourbaki}, $\id_A$ is in
  the image of this map if and only if $A$ is a projective module. In
  that case $H_2(\pi; A^w)$ vanishes and we are done, therefore we can
  assume $\id_A$ is not in the image.

  Since $A$ is reflexive, we have a short exact sequence $A
  \xrightarrow{j} F \xrightarrow{q} F/A$ for some free module
  $F$. This induces the commutative diagram
  \begin{equation*}
    \begin{tikzcd}
      A\otimes_{\mathbb{Z}\pi}A^{*} \arrow[r] \arrow[d]
      & F \otimes_{\mathbb{Z}\pi} A^{*} \arrow[d, "\cong"] \\
      \Hom_{\mathbb{Z}\pi}(A, A) \arrow[r, "j\circ-"]
      & \Hom_{\mathbb{Z}\pi}(A, F),
    \end{tikzcd}
  \end{equation*}
  where the right vertical map is an isomorphism and the image of the
  left vertical map contains $k\id_A$ but not $\id_A$. Hence $kj$ is
  in the image of the top map but $j$ is not, and so $j$ maps to a
  nontrivial $k$-torsion element $z$ in $F/A\otimes_{\mathbb{Z}\pi} A^{*}$.

  But now recall that $A^{*}$ is isomorphic to $\ker \partial_2$ in the chain
  complex $(C_{*}, \partial_{*})$, so the short exact sequence $A^{*}
  \xrightarrow{\iota} C_2 \to \ker \partial_1$ induces the exact sequence
    \[
    0 \to \Tor_1^{\mathbb{Z}\pi}(F/A,\ker \partial_1) \to F/A\otimes_{\mathbb{Z}\pi} A^{*} \to F/A\otimes_{\mathbb{Z}\pi} C_2. 
  \]
  The image of $z$ in $F/A\otimes_{\mathbb{Z}\pi} C_2 \cong \Hom_{\mathbb{Z}\pi}(C^2, F/A)$ is equal
  to the image of $j$ under the map $\Hom_{\mathbb{Z}\pi}(A, F) \cong
  \Hom_{\mathbb{Z}\pi}(A^{**}, F) \xrightarrow{q\circ-\circ\iota^{*}} \Hom_{\mathbb{Z}\pi}(C^2, F/A)$,
  but this is trivial since $q\circ j = 0$. Hence $z$ is in the kernel of
  the right map in the sequence, which by dimension shifting is
  isomorphic to $\Tor_1^{\mathbb{Z}\pi}(F/A,\ker \partial_1) \cong \Tor_2^{\mathbb{Z}\pi}(F/A, \ker
  \partial_0) \cong \Tor_3(F/A, \mathbb{Z}) \cong H_2(\pi; A^w)$. Therefore $k$ is equal to the
  order of $H_2(\pi; A^w)$, and we are done.
\end{proof}

The following is then immediate.

\begin{corollary}
  \label{flourish}
  Let $M$ and $M'$ be Poincar\'e $4$-complexes with fundamental group
  $\pi$, isomorphic quadratic $2$-types and $\pi_2(M) \cong \pi_2(M')$ stably
  isomorphic to $A\oplus A^{*}$. Suppose that $A$ is reflexive and that
  $H_2(\pi; A^w)$ is trivial or cyclic of prime order. Let $f \colon M \to
  B$ and $f' \colon M' \to B$ be $3$-connected maps to the Postnikov
  $2$-type $B$. Then the element $f_{*}[M] - f_{*}'[M'] \in H_4(B; \mathbb{Z})$
  is in the image of $\varphi_B \colon \mathbb{Z} \otimes_{\mathbb{Z}\pi} H_4(B; \mathbb{Z}\pi) \to H_4(B; \mathbb{Z}\pi)$. \qed
\end{corollary}

\begin{remark}
  The assumption that $H_2(\pi; A^w)$ is trivial or cyclic of prime
  order was chosen because it is satisfied by the groups $\mathbb{Z}/p\rtimes\mathbb{Z}$, as
  we show in \cref{histar-cyclic}. All of the results in this section
  also hold under the slightly different assumption that $H^1(\pi;
  \mathbb{Z}\pi)=0$, i.e.\ they also hold for one-ended groups. The proofs
  require minimal modification. Indeed in this case $\coker\partial^1 \cong
  \im\partial^2$ and so we have the short exact sequence $\coker\partial^1 \to C^2 \to
  \coker\partial^2 = A$, whence $H_2(\pi; A^w) \cong H_1(\pi; (\coker\partial^1)^w)$. Then
  since $C^0 = \mathbb{Z}\pi$, by the exact sequence
  \[
    0\to H_1(\pi;(\coker\partial^1)^w)\to\mathbb{Z}\cong H_0(\pi;C^0) \to H_0(\pi;C^1)
  \]
  we can see that $H_2(\pi; A^w)$ is infinite cyclic. The only place
  where finiteness is necessary in the proofs is in \cref{nontriv-im},
  where a certain differential with codomain $H_0(\pi; H^1(\pi; \mathbb{Z}\pi))$ was
  shown to be trivial, which clearly also holds if $H^1(\pi; \mathbb{Z}\pi)=0$. The
  rest of the proofs work equally well for infinite cyclic $H_2(\pi;
  A^w)$.
\end{remark}

\section{Properties of $\mathbb{Z}/p \rtimes \mathbb{Z}$}
\label{sec-prop}

Now we specialize to the fundamental groups $\mathbb{Z}/p \rtimes \mathbb{Z}$. From now on,
let $p$ be an odd prime, let $\pi = \mathbb{Z}/p \rtimes \mathbb{Z}$ where $\mathbb{Z}$ acts on $\mathbb{Z}/p$ by
multiplication by $q$, let $w \colon \pi \to \{\pm1\}$ be a homomorphism,
and set:
\begin{itemize}
\item $t, T$ generators of $\pi$ such that $\pi = \langle\,t, T \mid tTt^{-1}T^{-q},
  T^p\,\rangle$;
\item $N = 1+T+T^2+\cdots+T^{p-1} \in \mathbb{Z}\pi$;
\item $\wh N = 1+T+T^2+\cdots+T^{q-1} \in \mathbb{Z}\pi$;
\item $A = \langle q-\ol t, N\rangle$, the left ideal in $\mathbb{Z}\pi$;
\item $A' = \langle q-t, 1-T\rangle$, the left ideal in $\mathbb{Z}\pi$.
\end{itemize}

In this section we will use \cref{stab-iso} to determine the possible
stable isomorphism classes of $\pi_2(M)$, and we will verify that $\pi$
has the desirable properties of \cref{sec-ker,sec-diff}.

\begin{lemma}
  \label{hfour}
  We have that
  \[
    H_4(\pi; \mathbb{Z}^w) \cong
    \begin{cases}
      \mathbb{Z}/p &\quad q^2 \equiv w(t) \mod p,\\
      0 &\quad \text{otherwise.}
    \end{cases}
  \]
\end{lemma}

\begin{proof}
  We can calculate group homology using the following free resolution for
  the trivial right $\mathbb{Z}\pi$-module $\mathbb{Z}$:
  \begin{equation*}
    \begin{tikzcd}[column sep=large]
      \cdots \arrow[r, "1-T"]
      & \mathbb{Z}\pi \arrow[r, "N"] \arrow[d, phantom, "\oplus"]
      & \mathbb{Z}\pi \arrow[r, "1-T"] \arrow[d, phantom, "\oplus"]
      & \mathbb{Z}\pi. \\
      \cdots \arrow[ur, sloped, "1-qt"', "\partial_3"] \arrow[r, "N"']
      & \mathbb{Z}\pi \arrow[ur, sloped, "\wh Nt-1"', "\partial_2"] \arrow[r, "1-T"']
      & \mathbb{Z}\pi \arrow[ur, sloped, "1-t"', "\partial_1"]
    \end{tikzcd}
  \end{equation*}
  Using this, $H_4(\pi; \mathbb{Z}^w)$ is given by the homology of the following
  chain complex:
  \begin{equation*}
    \begin{tikzcd}[column sep=huge]
      \mathbb{Z} \arrow[r, "0"] \arrow[d, phantom, "\oplus"]
      & \mathbb{Z} \arrow[r, "p"] \arrow[d, phantom, "\oplus"]
      & \mathbb{Z} \arrow[d, phantom, "\oplus"] \\
      \mathbb{Z} \arrow[ur, sloped, "1-w(t)q^2"', "\partial_5\otimes\id"] \arrow[r, "p"']
      & \mathbb{Z} \arrow[r, "0"'] \arrow[ur, sloped, "w(t)q^2-1"', "\partial_4\oplus\id"]
      & \mathbb{Z}.
    \end{tikzcd}
  \end{equation*}
  Going through the cases gives the statement.
\end{proof}

\begin{lemma}
  \label{stab-class}
  Let $M$ be a Poincar\'e $4$-complex with fundamental group $\pi =
  \mathbb{Z}/p\rtimes\mathbb{Z}$, orientation character $w$, and classifying map $c\colon M \to
  B\pi$. Then $\pi_2(M)$ is stably isomorphic to $A \oplus A'$ if $c_{*}[M] = 0
  \in H_4(\pi; \mathbb{Z}^w)$ and is stably free otherwise.
\end{lemma}

\begin{proof}
  By \cref{stab-iso}, for fixed $p$, $q$, and $w$, the stable
  isomorphism class of $\pi_2(M)$ depends only on $c_{*}[M]$.

  First let $c_{*}[M] = 0$. As in \cref{sec-ker}, in this case
  $\pi_2(M)$ is stably isomorphic to $\ker\partial_2 \oplus \coker\partial^2$, where
  $(C_{*}, \partial_{*})$ is the following free resolution for the trivial
  left $\mathbb{Z}\pi$-module $\mathbb{Z}$:
  \begin{equation*}
    \begin{tikzcd}[column sep=large]
      \cdots \arrow[r, "1-T"]
      & \mathbb{Z}\pi \arrow[r, "N"] \arrow[d, phantom, "\oplus"]
      & \mathbb{Z}\pi \arrow[r, "1-T"] \arrow[d, phantom, "\oplus"]
      & \mathbb{Z}\pi. \\
      \cdots \arrow[ur, sloped, "q-t"', "\partial_3"] \arrow[r, "N"']
      & \mathbb{Z}\pi \arrow[ur, sloped, "t-\wh N"', "\partial_2"] \arrow[r, "1-T^q"']
      & \mathbb{Z}\pi \arrow[ur, sloped, "1-t"', "\partial_1"]
    \end{tikzcd}
  \end{equation*}

  The summand $\ker\partial_2 \cong \coker\partial_4$ can be seen to be isomorphic to
  the left ideal $A'$ by the exact sequence
  \begin{equation*}
    \begin{tikzcd}[column sep=large]
      \mathbb{Z}\pi \arrow[r, "N"] \arrow[d, phantom, "\oplus"]
      & \mathbb{Z}\pi \arrow[r, "1-T"] \arrow[d, phantom, "\oplus"]
      & \mathbb{Z}\pi. \\
      \mathbb{Z}\pi \arrow[ur, sloped, "t-q\wh N"', "\partial_4"] \arrow[r,
      "1-T^q"']
      & \mathbb{Z}\pi \arrow[ur, sloped, "q-t"']
    \end{tikzcd}
  \end{equation*}
  
  Similarly, the summand $\coker\partial^2$ can be seen to be isomorphic
  to the left ideal $A$ by the following exact sequence:
  \begin{equation*}
    \begin{tikzcd}[column sep=large]
      \mathbb{Z}\pi \arrow[r, "1-T^{-q}"] \arrow[d, phantom, "\oplus"]
      & \mathbb{Z}\pi \arrow[r, "N"] \arrow[d, phantom, "\oplus"]
      & \mathbb{Z}\pi. \\
      \mathbb{Z}\pi \arrow[ur, sloped, "\ol{t-\wh N}"', "\partial^2"]
      \arrow[r, "N"']
      & \mathbb{Z}\pi \arrow[ur, sloped, "q-\ol t"']
    \end{tikzcd}
  \end{equation*}

  Now consider nontrivial $c_{*}[M]$. By \cref{hfour} this can
  happen only if $q^2 \equiv w(t) \mod p$.

  Let $L_l$ denote the lens space $L(p, l)$. By \cite[Theorem~V]{olum}
  we have that, for any $l$, there is a self-map of degree $w(t)$ on
  $L_l$ that induces multiplication by $q$ on the fundamental
  group. By Poincar\'e duality, this map is a self homotopy
  equivalence. Hence the mapping torus $T_l$ is a Poincar\'e
  $4$-complex with fundamental group $\pi$ and vanishing $\pi_2$. In
  particular, if $q=1$, we have that $T_l$ is just the product $S^1 \times
  L_l$.

  Consider the subgroup $\pi' = \langle T, t^{p-1}\rangle$ of index $p-1$. This is
  isomorphic to the direct product $\mathbb{Z}/p \times \mathbb{Z}$. The induced map $H_4(\pi';
  \mathbb{Z}) \to H_4(\pi;\mathbb{Z}^w)$ is an isomorphism, and the cover of $T_l$ with
  respect to this subgroup is $S^1\times L_l$, hence the preimage of
  $c_{*}[T_l]$ in $H_4(\pi'; \mathbb{Z})$ is equal to $c_{*}[S^1\times L_l]$. By
  \cite[Proposition~9.2]{KPT-2}, we have that any nonzero element of
  $H_4(\pi'; \mathbb{Z})$ is equal to $c_{*}[S^1\times L_l]$ for some $l$, and so
  $c_{*}[M] = c_{*}[T_l]\in H_4(\pi; \mathbb{Z}^w)$. By \cref{stab-iso} $\pi_2(M)$ is
  stably isomorphic to $\pi_2(T_l) \cong 0$, so $\pi_2(M)$ is indeed stably
  free in this case.
\end{proof}

The next two lemmas verify that $\pi$ satisfies the conditions which
allow us to apply the results from \cref{sec-ker,sec-diff}.

\begin{lemma}
  \label{dualiso}
  The ideals $A$ and $A'$ are dual to each other. In particular, they
  are both reflexive.
\end{lemma}

\begin{proof}
  We have that $A'$ is the dual of $A$ by \cite[Lemma~3.12]{KPT}. On
  the other hand by \cite[Lemma~3.13]{KPT} there exists an exact
  sequence
  \[
    0 \to H^2(\pi; \mathbb{Z}\pi) \to A \to (A')^{*} \to H^3(\pi; \mathbb{Z}\pi) \to 0.
  \]
  By Shapiro's Lemma $H^2(\pi; \mathbb{Z}\pi)$ and $H^3(\pi; \mathbb{Z}\pi)$ are both trivial,
  hence $A$ is the dual of $A'$.
\end{proof}

\begin{lemma}
  \label{histar-cyclic}
  The group $H_2(\pi; A^w)$ is trivial or cyclic of order $p$.
\end{lemma}

\begin{proof}
  Let $K$ be a presentation comlpex and recall that $H^2(K; \mathbb{Z}\pi)$ is
  stably isomorphic to $A$. Consider the $\mathbb{Z}$ subgroup of $\pi$ of index
  $p$. There is a transfer map on group homology such that the
  composition
  \[
    H_2(\pi; A^w) \xrightarrow{\tr} H_2(\mathbb{Z}; A^w) \to H_2(\pi; A^w)
  \]
  is given by multiplication by $p$. As a $\mathbb{Z}\mathbb{Z}$-module $A$ is stably
  isomorphic to $H^2(\wh K; \mathbb{Z}\mathbb{Z})$, where $\wh K$ is the
  $p$-sheeted cover of $K$ of corresponding to the $\mathbb{Z}$
  subgroup. However $\wh K$ is still a finite $2$-complex, hence
  by \cite[Lemma~1.8]{KPT} the group $H^2(\wh K; \mathbb{Z}\mathbb{Z})$ is stably
  isomorphic to the dual of the augmentation ideal in $\mathbb{Z}\mathbb{Z}$, which is
  free. Hence we have that $H_2(\mathbb{Z}; A^w) = 0$ and the group $H_2(\pi;
  A^w)$ is annihilated by $p$.

  Next, let $(C_{*}, \partial_{*})$ be the free resolution of $\mathbb{Z}$ from
  before. Tensoring the short exact sequence $0 \to A' \to C_2 \to \ker \partial_1
  \to 0$ with the quotient $\mathbb{Z}\pi/A$ gives the exact sequence
  \[
    0 \to \Tor_1^{\mathbb{Z}\pi}(\ker\partial_1, \mathbb{Z}\pi/A) \to A'\otimes_{\mathbb{Z}\pi} \mathbb{Z}\pi/A \to C_2\otimes_{\mathbb{Z}\pi} \mathbb{Z}\pi/A.
  \]
  We have that $A' \otimes_{\mathbb{Z}\pi} \mathbb{Z}\pi/A$ is generated by the elements $(q-t)\otimes1$
  and $(1-T)\otimes1$. The former maps trivially to $C_2\otimes_{\mathbb{Z}\pi}\mathbb{Z}\pi/A$, hence
  $\Tor_1^{\mathbb{Z}\pi}(\ker\partial_1, \mathbb{Z}\pi/A)$ is generated by $(q-t)\otimes1$. But now by
  dimension shifting we have that $\Tor_1^{\mathbb{Z}\pi}(\ker\partial_1, \mathbb{Z}\pi/A) \cong
  \Tor_2(\ker\partial_0, \mathbb{Z}\pi/A) \cong \Tor_3(\mathbb{Z}^w, \mathbb{Z}\pi/A) \cong H_2(\pi; A^w)$. The
  twisting is introduced by the conversion of the left-module short
  exact sequence $\ker\partial_0 \to C_0 \to \mathbb{Z}$ into right modules. So we have
  that $H_2(\pi; A^w)$ is cyclic, but previously we showed that it is
  $p$-torsion, so we are done.
\end{proof}

\section{The classification for stably free $\pi_2(M)$}
\label{sec1}

We now begin the proof of \cref{mainth}. The proof is split into two
cases, depending on the stable isomorphism class of $\pi_2(M)$ according
to \cref{stab-class}. In this section we prove the easier case, when
$\pi_2(M)$ is stably free; the other case is proved in \cref{sec2}.

\begin{lemma}
  \label{main-free}
  Let $p$ be an odd prime, and let $M$ and $M'$ be Poincar\'e
  $4$-complexes with isomorphic quadratic $2$-types, where the
  fundamental group is $\pi = \mathbb{Z}/p \rtimes \mathbb{Z}$, the orientation character is
  $w$, and $\pi_2(M) \cong \pi_2(M')$ is stably free. Let $c \colon M \to B\pi$
  and $c' \colon M' \to B\pi$ be classifying maps. If $c_{*}[M] =
  c_{*}'[M'] \in H_4(\pi; \mathbb{Z}^w)$, then $M$ and $M'$ are homotopy equivalent
  over their (common) Postnikov $2$-type $B$.
\end{lemma}

\begin{proof}
  Consider the Serre spectral sequence of the fibration $\wt B
  \to B \to B\pi$. This has $E^2$ page $H_k(\pi; H_l(\wt B; \mathbb{Z})^w)$ and
  converges to $H_{k+l}(B; \mathbb{Z}^w)$. Since $H_2(\wt B; \mathbb{Z}) \cong
  \pi_2(M)$ is stably free, the only nontrivial terms on the
  $4$-diagonal are $H_0(\pi; H_4(\wt B; \mathbb{Z})^w) \cong \mathbb{Z}^w\otimes_{\mathbb{Z}\pi}H_4(B;
  \mathbb{Z}\pi)$ and $H_4(\pi; \mathbb{Z}^w)$. Since the map $\varphi_B \colon \mathbb{Z}^w\otimes_{\mathbb{Z}\pi}H_4(B;
  \mathbb{Z}\pi) \to H_4(B; \mathbb{Z}^w)$ is the edge homomorphism, we obtain the exact
  sequence
  \[
    \mathbb{Z}^w\otimes_{\mathbb{Z}\pi}H_4(B; \mathbb{Z}\pi)\xrightarrow{\varphi_B}H_4(B; \mathbb{Z}^w)\to H_4(\pi; \mathbb{Z}^w) \to 0.
  \]
  Let $f \colon M \to B$ and $f' \colon M' \to B$ be the maps to the
  Postnikov $2$-type. Then $f_{*}[M] - f_{*}'[M'] \in H_4(B; \mathbb{Z}^w)$ maps
  to $c_{*}[M] - c_{*}'[M'] \in H_4(\pi; \mathbb{Z}^w)$.

  Now by \cref{comm-diag} we fit the above exact sequence into a
  commutative diagram
  \begin{equation*}
    \begin{tikzcd}
      \mathbb{Z}^w \otimes_{\mathbb{Z}\pi} H_4(B; \mathbb{Z}\pi) \arrow[r, "\varphi_B"] \arrow[d, "\mathcal{B}_{H_2(B;, \mathbb{Z}\pi)}"]
      & H_4(B; \mathbb{Z}^w) \arrow[d, "\Theta_B"] \arrow[r]
      & H_4(\pi; \mathbb{Z}^w) \arrow[r]
      & 0 \\
      \Her^w(H_2(B; \mathbb{Z}\pi)^{*}) \arrow[r, "\ev^{*}"]
      & \Her^w(H^2(B; \mathbb{Z}\pi)).
    \end{tikzcd}
  \end{equation*}
  The map $\mathord{\ev^{*}} \circ \mathcal{B}_{H_2(B; \mathbb{Z}\pi)}$ is injective by
  \cref{B-inj} and \cite[Lemma~5.1]{KPR}, and therefore so is the map
  $\varphi_B$. Recalling that the element $f_{*}[M] - f_{*}'[M']$ is in the
  kernel of $\Theta_B$, a diagram chase shows that $c_{*}[M] = c_{*}'[M'] \in
  H_4(\pi; \mathbb{Z}^w)$ implies $f_{*}[M] = f_{*}'[M'] \in H_4(B; \mathbb{Z}^w)$, so the
  theorem follows by \cref{bauesbleile}.
\end{proof}

\begin{theorem}
  \label{form-dist-free}
  Let $p$ be an odd prime, and let $M$ and $M'$ be a Poincar\'e
  $4$-complexes with fundamental group $\mathbb{Z}/p \rtimes \mathbb{Z}$ and isomorphic
  quadratic $2$-types. Suppose that $\pi_2(M) \cong \pi_2(M')$ is stably
  free. If $\lambda_M^{\mathbb{Z}/p} \cong \lambda_{M'}^{\mathbb{Z}/p}$, then $M$ and $M'$ are homotopy
  equivalent.
\end{theorem}

\begin{proof}
  By \cref{main-free}, it suffices to show that $c_{*}[M] = c_{*}'[M']
  \in H_4(\pi; \mathbb{Z}^w)$. This is trivial if the group $H_4(\pi; \mathbb{Z}^w)$ vanishes,
  hence we may assume that $q^2 \equiv w(t) \mod p$ by \cref{hfour}. Our
  notion of isomorphism $\lambda_M^{\mathbb{Z}/p} \cong \lambda_{M'}^{\mathbb{Z}/p}$ ensures that the
  two forms induce the same form
  \[
    \lambda_{\pi}^{\mathbb{Z}/p} \colon H^2(\pi; (\mathbb{Z}/p)^q) \times H^2(\pi; (\mathbb{Z}/p)^q) \to \mathbb{Z}/p,
  \]
  i.e.\ we have that $(\alpha\cup\beta)\cap c_{*}[M] = (\alpha\cup\beta)\cap c_{*}'[M']$ for all
  $\alpha,\beta \in H^2(\pi; (\mathbb{Z}/p)^q)$. So we can conclude that $c_{*}[M] =
  c_{*}'[M']$ as desired if we show that the cup product $H^2(\pi;
  (\mathbb{Z}/p)^q) \times H^2(\pi; (\mathbb{Z}/p)^q) \to H^4(\pi; (\mathbb{Z}/p)^w)$ is nonzero.

  For this, consider the subgroup $\pi' = \langle T, t^{p-1}\rangle$ of index
  $p-1$. This is isomorphic to the direct product $\mathbb{Z}/p\times\mathbb{Z}$, and so
  $H^2(\pi'; \mathbb{Z}/p) = \mathbb{Z}/p\oplus\mathbb{Z}/p$ by the K\"unneth theorem. There is a
  transfer map on group cohomology such that the composition
  \[
    H^2(\pi'; \mathbb{Z}/p) \xrightarrow{\tr} H^2(\pi; (\mathbb{Z}/p)^q) \to H^2(\pi'; \mathbb{Z}/p)
  \]
  is given by multiplication by $p-1$, and so the map $H^2(\pi; (\mathbb{Z}/p)^q)
  \to H^2(\pi'; \mathbb{Z}/p)$ is surjective. Again by the K\"unneth theorem the
  cup product $H^2(\pi'; \mathbb{Z}/p)\times H^2(\pi'; \mathbb{Z}/p) \to H^4(\pi'; \mathbb{Z}/p)$ is nonzero,
  and the proof is complete by naturality of the cup product.
\end{proof}

\section{The classification for non--stably free $\pi_2(M)$}
\label{sec2}

In this section we deal with the case when $c_{*}[M] = 0 \in H_4(\pi;
\mathbb{Z}^w)$ and $\pi_2(M)$ is stably isomorphic to the sum $A\oplus A'$. The main
theorem in this case will be shown by applying \cref{realization-cor}
to obtain $p$ many different homotopy types, and then showing that the
form $\lambda_M^{\mathbb{Z}/p}$ detects these. The first step is to calculate the
kernel of the map $\mathcal{B}_{\pi_2(M)}$.

\subsection{The kernel of $\mathcal{B}_{\pi_2(M)}$}

\begin{lemma}[cf.\ {\cite[Lemma~6.12]{KPR}}]
  \label{ker-B-stab-inv}
  The kernel of the map $\mathcal{B}_{\pi_2(M)}$ is isomorphic to the kernel of
  the map $\mathcal{B}_{A\oplus A'}$.
\end{lemma}

\begin{proof}
  Recall that $\pi_2(M)\oplus\mathbb{Z}\pi^j \cong A\oplus A'\oplus\mathbb{Z}\pi^k$. Call this module $L$. Both
  $A$ and $A'$ are reflexive, in particular torsionless, so they embed
  into a free module, and therefore so does $\pi_2(M)$.

  On one hand, by \cref{dsum-rem} we know $\mathcal{B}_L$ decomposes as the sum
  of the maps $\mathcal{B}_{A\oplus A'}$, $\mathcal{B}_{\mathbb{Z}\pi^k}$ and $k$ copies of $\ev_{A\oplus
    A'}$. The map $\mathcal{B}_L$ is injective by \cref{B-inj}, and $\ev_{A\oplus
    A'}$ is injective since $A\oplus A'$ is torsionless, so we have
  $\ker\mathcal{B}_L \cong \ker\mathcal{B}_{A\oplus A'}$. But the same argument also holds for the
  decomposition of $\mathcal{B}_L$ as the sum of $\mathcal{B}_{\pi_2(M)}$, $\mathcal{B}_{\mathbb{Z}\pi^j}$, and
  $\ev_{\pi_2(M)}$, so we are done.
\end{proof}

\begin{lemma}
  \label{rkzero}
  Let $J$ be either $A$ or $A'$. Then the kernel of the map $\mathcal{B}_J
  \colon \mathbb{Z}^w \otimes_{\mathbb{Z}\pi} \Gamma(J) \to \Her^w(J^{*})$ coincides with the torsion
  subgroup of $\mathbb{Z}^w \otimes_{\mathbb{Z}\pi} \Gamma(J)$.
\end{lemma}

\begin{proof}
  First note the commutative diagram
  \begin{equation*}
    \begin{tikzcd}
      \mathbb{Z}^w \otimes_{\mathbb{Z}\pi} \Gamma(J) \arrow[r, "\mathcal{B}_J"] \arrow[d]
      & \Her^w(J^{*}) \arrow[d] \\
      \mathbb{Z}^w \otimes_{\mathbb{Z}\pi} \Gamma(\mathbb{Z}\pi) \arrow[r, "\mathcal{B}_{\mathbb{Z}\pi}"]
      & \Her^w(\mathbb{Z}\pi^{*}),
    \end{tikzcd}
  \end{equation*}
  where the vertical maps are induced by the inclusion $J \to \mathbb{Z}\pi$ of the
  ideal $J$. The bottom map is injective by \cref{B-inj}, hence the
  kernel of $\mathcal{B}_J$ is contained in the kernel of the left map. Since
  $\Her^w(J^{*})$ is torsion-free, to prove the lemma it suffices to
  show that the kernel of the left map coincides with the torsion
  subgroup of $\mathbb{Z}^w\otimes_{\mathbb{Z}\pi}\Gamma(J)$.

  The kernel of that map is isomorphic to $H_1(\pi; (\Gamma(\mathbb{Z}\pi)/\Gamma(J))^w)$ by
  the $\Tor$ long exact sequence since $\Gamma(\mathbb{Z}\pi)$ is free by
  \cref{Gamma-zpi-free}. Also $\mathbb{Z}^w \otimes_{\mathbb{Z}\pi} \Gamma(\mathbb{Z}\pi)$ is torsion-free, so
  it is enough to show that $H_1(\pi; (\Gamma(\mathbb{Z}\pi)/\Gamma(J))^w)$ is torsion.

  Consider the Hochschild--Serre spectral sequence associated to the
  short exact sequence of groups $0 \to \mathbb{Z}/p \to \pi \to \mathbb{Z} \to 0$. The five-term
  exact sequence gives the following short exact sequence:
  \[
    H_1(\mathbb{Z}/p; \Gamma(\mathbb{Z}\pi)/\Gamma(J)) \to H_1(\pi; (\Gamma(\mathbb{Z}\pi)/\Gamma(J))^w) \to H_1(\mathbb{Z};
    (\Gamma(\mathbb{Z}\pi)/\Gamma(J))^w) \to 0.
  \]
  But the homology of $\mathbb{Z}/p$ is torsion, therefore it suffices to show
  that $H_1(\mathbb{Z}, (\Gamma(\mathbb{Z}\pi)/\Gamma(J))^w)$ is trivial, i.e.\ that the map
  $\mathbb{Z}^w\otimes_{\mathbb{Z}\mathbb{Z}}\Gamma(J) \to \mathbb{Z}^w\otimes_{\mathbb{Z}\mathbb{Z}}\Gamma(\mathbb{Z}\pi)$ is injective.

  For this, consider the cases for $J$ separately. First let $J =
  A'$. Then we have the short exact sequence
  \[
    0 \to A' \to \mathbb{Z}\pi \to \mathbb{Z}[1/q] \to 0.
  \]
  Considering the structure of the terms in this sequence as
  $\mathbb{Z}\mathbb{Z}$-modules, we see that the term on the right is isomorphic to
  $\mathbb{Z}\mathbb{Z}/\langle q-t\rangle$ and the middle term is isomorphic to $\mathbb{Z}\mathbb{Z}^p$, with the
  map a direct sum of $p-1$ zero maps and the obvious quotient
  map. Hence we have that $A' \cong_{\mathbb{Z}\mathbb{Z}} \langle q-t\rangle \oplus \mathbb{Z}\mathbb{Z}^{p-1}$.
  
  By \cref{dsum-rem}, the induced map on $\Gamma$ groups decomposes as the
  following sum:
  \begin{equation*}
    \begin{tikzcd}[column sep=small]
      \Gamma(\mathbb{Z}\mathbb{Z}^{p-1}) \arrow[r, phantom, "\oplus"] \arrow[d]
      & \Gamma(\langle q-t\rangle) \arrow[r, phantom, "\oplus"] \arrow[d]
      & \mathbb{Z}\mathbb{Z}^{p-1} \otimes_{\mathbb{Z}} \langle q-t\rangle \arrow[d] \\
      \Gamma(\mathbb{Z}\mathbb{Z}^{p-1}) \arrow[r, phantom, "\oplus"]
      & \Gamma(\mathbb{Z}\mathbb{Z}) \arrow[r, phantom, "\oplus"]
      & \mathbb{Z}\mathbb{Z}^{p-1} \otimes_{\mathbb{Z}} \mathbb{Z}\mathbb{Z}.
    \end{tikzcd}
  \end{equation*}
  The left and right maps are obviously injective after tensoring. It
  remains to show that the map $\mathbb{Z}^w \otimes_{\mathbb{Z}\mathbb{Z}} \Gamma(\langle q-t\rangle) \to \mathbb{Z}^w \otimes_{\mathbb{Z}\mathbb{Z}}
  \Gamma(\mathbb{Z}\mathbb{Z})$ induced by the inclusion is injective.

  The element $q-t \in \mathbb{Z}\mathbb{Z}$ is not a zero divisor, hence $\langle q-t\rangle$ is
  isomorphic to $\mathbb{Z}\mathbb{Z}$. By \cref{Gamma-zpi-free}, we know that $\Gamma(\mathbb{Z}\mathbb{Z})$
  is free with basis $\{1\otimes1, 1\otimes t +t\otimes1, \cdots\}$, and we can describe the
  map $\Gamma(\langle q-t\rangle) \to \Gamma(\mathbb{Z}\mathbb{Z})$ with the following infinite matrix with
  respect to this basis:
  \[
    \begin{pmatrix}
      q^2 + t & -2qt    &       &     & \\
           -q & q^2+t   & -qt   &     & \\
              & -q      & q^2+t & -qt & \\
              &         & \ddots     & \ddots   & \ddots
    \end{pmatrix}.
  \]
  Clearly, this is injective after tensoring, and the proof of this
  case is finished.

  Now let $J = A$. In this case, we have the short exact sequence
  \[
    0 \rightarrow A \rightarrow \mathbb{Z}\pi \rightarrow \mathbb{Z}\pi/A \rightarrow 0
  \]
  of $\mathbb{Z}\mathbb{Z}$-modules. Here the middle term is still $\mathbb{Z}\mathbb{Z}^p$, but the term
  on the right is isomorphic to $(\mathbb{Z}\mathbb{Z}/\langle q-\ol t\rangle)^{p-1}$, with the
  map a direct sum of $p-1$ quotient maps and one zero map. Hence we have
  that $A \cong_{\mathbb{Z}\mathbb{Z}} \langle q-\ol t\rangle^{p-1}\oplus\mathbb{Z}\mathbb{Z}$. As before, we show
  that the induced map $\mathbb{Z}^w \otimes_{\mathbb{Z}\mathbb{Z}} \Gamma(A) \to \mathbb{Z}^w \otimes_{\mathbb{Z}\mathbb{Z}} \Gamma(\mathbb{Z}\pi)$ is
  injective.

  The induced map on $\Gamma$ groups now has the form:
  \begin{equation*}
    \begin{tikzcd}[column sep=small]
      \Gamma(\langle q-\ol t\rangle^{p-1}) \arrow[r, phantom, "\oplus"] \arrow[d]
      & \Gamma(\mathbb{Z}\mathbb{Z}) \arrow[r, phantom, "\oplus"] \arrow[d]
      & \langle q-\ol t\rangle^{p-1} \otimes_{\mathbb{Z}} \mathbb{Z}\mathbb{Z} \arrow[d] \\
      \Gamma(\mathbb{Z}\mathbb{Z}^{p-1}) \arrow[r, phantom, "\oplus"]
      & \Gamma(\mathbb{Z}\mathbb{Z}) \arrow[r, phantom, "\oplus"]
      & \mathbb{Z}\mathbb{Z}^{p-1} \otimes_{\mathbb{Z}} \mathbb{Z}\mathbb{Z}.
    \end{tikzcd}
  \end{equation*}
  The second and third maps here are obviously injective after
  tensoring, so it remains to show the first map is injective after
  tensoring. The map $\langle q-\ol t\rangle\otimes_{\mathbb{Z}\mathbb{Z}}\langle q-\ol t\rangle \to
  \mathbb{Z}\mathbb{Z}\otimes_{\mathbb{Z}\mathbb{Z}}\mathbb{Z}\mathbb{Z}$ is injective because $\mathbb{Z}\mathbb{Z}$ has no zero divisors, so by
  induction it remains to show that $\mathbb{Z}^w\otimes_{\mathbb{Z}\mathbb{Z}}\Gamma(\langle
  q-\ol t\rangle)\to\mathbb{Z}^w\otimes_{\mathbb{Z}\mathbb{Z}}\Gamma(\mathbb{Z}\mathbb{Z})$ is injective. This works similarly to
  the case above.
\end{proof}

\begin{lemma}
  \label{weird-ideal-B-inj}
  The group $\mathbb{Z}^w \otimes_{\mathbb{Z}\pi} \Gamma(A')$ is torsion-free.
\end{lemma}

\begin{proof}
  Let $K$ be a presentation complex for the group $\pi$. Then we have
  that $\pi_2(K)$ is stably isomorphic to $A'$. Consider the $\mathbb{Z}$
  subgroup of $\pi$ of index $p$. There is a transfer map on group
  homology such that the composition
  \[
    \mathbb{Z}^w\otimes_{\mathbb{Z}\pi}\Gamma(A') \xrightarrow{\tr} \mathbb{Z}^w\otimes_{\mathbb{Z}\mathbb{Z}}\Gamma(A') \to \mathbb{Z}^w\otimes_{\mathbb{Z}\pi}\Gamma(A')
  \]
  is given by multiplication by $p$. As a $\mathbb{Z}\mathbb{Z}$-module $A'$ is stably
  isomorphic to $\pi_2(\wt K)$, where $\wt K$ is the
  $p$-sheeted cover of $K$ corresponding to the $\mathbb{Z}$
  subgroup. However $\wt K$ is still a finite $2$-complex,
  hence by \cite[Lemma~1.8]{KPT} $\pi_2(\wt K)$ is stably
  isomorphic to the augmentation ideal in $\mathbb{Z}\mathbb{Z}$, which is free, i.e.\
  $A'$ is stably free. Thus $\mathbb{Z}^w\otimes_{\mathbb{Z}\mathbb{Z}}\Gamma(A')$ is a direct summand of
  $\mathbb{Z}^w\otimes_{\mathbb{Z}\mathbb{Z}}\Gamma(\mathbb{Z}\mathbb{Z}^k)$, and is therefore torsion-free by
  \cref{tors-free}. But now any torsion element in $\mathbb{Z}^w\otimes_{\mathbb{Z}\pi}\Gamma(A')$
  must vanish under the composition above. Since $p$ is prime, this
  means that $\mathbb{Z}^w\otimes_{\mathbb{Z}\pi}\Gamma(A')$ can only contain $p$-torsion.

  Similarly, consider the subgroup $\pi' = \langle T, t^{p-1}\rangle$ of index
  $p-1$. Note that this is isomorphic to the direct product $\mathbb{Z}/p \times
  \mathbb{Z}$. Hence the cover $\ol K$ of $K$ corresponding to this
  subgroup is still a finite $2$-complex and so has $\pi_2(\ol K)$
  stably isomorphic to $L$ the augmentation ideal in $\pi'$ (since the
  ideal $A'$ is precisely the augmentation ideal whenever $q=1$). As
  above, $\mathbb{Z}^w\otimes_{\mathbb{Z}\pi'}\Gamma(A')$ is a direct summand of
  $\mathbb{Z}^w\otimes_{\mathbb{Z}\pi'}\Gamma(L\oplus\mathbb{Z}\pi'^k)$. The latter decomposes as the direct sum of a
  free abelian group and $\mathbb{Z}^w\otimes_{\mathbb{Z}\pi'}\Gamma(L)$. This is torsion-free as it
  injects into $\Her^w(L^{*})$ by \cite[Corollary~6.7]{KPR}. Again
  considering the composition with the transfer map, we obtain that
  any torsion in $\mathbb{Z}^w\otimes_{\mathbb{Z}\pi}\Gamma(A')$, which we before showed must be
  $p$-torsion, must be annihilated by $p-1$, which is
  absurd. Therefore, $\mathbb{Z}^w\otimes_{\mathbb{Z}\pi}\Gamma(A')$ has no torsion.
\end{proof}

We wish to use a similar argument involving the transfer map to
compute the torsion in $\mathbb{Z}^w\otimes_{\mathbb{Z}\pi}\Gamma(A)$, but first we need to know the
torsion in the case when $q=1$.

Recall that we use $a\sq b$ as shorthand for $a\otimes b + b\otimes a$.

\begin{lemma}
  \label{torsfin-dir}
  Suppose $q = 1$, so that $\pi$ is the direct product $\mathbb{Z}/p \times \mathbb{Z}$. If $w$
  is nontrivial, the group $\mathbb{Z}^w\otimes_{\mathbb{Z}\pi}\Gamma(A)$ is torsion-free; otherwise,
  its torsion subgroup is isomorphic to $\mathbb{Z}/p$ and is generated by the
  image of the element
  \[
    \rho = N\sq(1-\ol t) - (1-\ol t)\otimes(1-\ol t) -
    (1-\ol t)\sq T(1-\ol t) - \cdots - (1-\ol t)\sq
    T^{\frac{p-1}{2}}(1-\ol t) \in \Gamma(A).
  \]
\end{lemma}

\begin{proof}
  Start with the short exact sequence of left $\mathbb{Z}\pi$-modules
  \[
    0 \to \mathbb{Z}\pi \to A \to \mathbb{Z} \to 0
  \]
  where the first map is given by $1 \mapsto 1-\ol t$ and $1 \in \mathbb{Z}$ lifts
  to $N \in A$.

  From this exact sequence, by \cref{exact-seqs} we get the
  $\Gamma$-functor short exact sequences
  \[
    0 \to \Gamma(\mathbb{Z}\pi) \to K \to \mathbb{Z}\pi \otimes \mathbb{Z} \to 0
  \]
  \[
    0 \to K \xrightarrow{f} \Gamma(A) \to \Gamma(\mathbb{Z}) \cong \mathbb{Z} \to 0.
  \]
  The first of these splits, hence $K \cong \Gamma(\mathbb{Z}\pi) \oplus \mathbb{Z}\pi$. Note that the map
  $f$ is given by $1 \mapsto (1-\ol t) \sq N$ on the $\mathbb{Z}\pi$
  summand.

  Tensoring the second $\Gamma$-functor short exact sequence with $\mathbb{Z}^w$
  over $\mathbb{Z}\pi$, we obtain the long exact sequence
  \[
    \cdots \to H_1(\pi; \mathbb{Z}^w) \xrightarrow{\delta} (\mathbb{Z}^w \otimes_{\mathbb{Z}\pi} \Gamma(\mathbb{Z}\pi)) \oplus \mathbb{Z}^w
    \xrightarrow{\id \otimes f} \mathbb{Z}^w \otimes_{\mathbb{Z}\pi} \Gamma(A) \to \mathbb{Z}^w \otimes_{\mathbb{Z}\pi} \mathbb{Z} \to 0, 
  \]
  where the first term comes from the isomorphism $\Tor_1^{\mathbb{Z}\pi}(\mathbb{Z}^w, \mathbb{Z})
  \cong H_1(\pi; \mathbb{Z}^w)$.

  We deal first with the case where $w$ is nontrivial. In this case
  $H_1(\pi; \mathbb{Z}^w) = 0$ and $\mathbb{Z}^w \otimes_{\mathbb{Z}\pi} \mathbb{Z} = \mathbb{Z}/2$. Since $\mathbb{Z}^w \otimes_{\mathbb{Z}\pi} \Gamma(\mathbb{Z}\pi) \oplus \mathbb{Z}^w$
  is torsion-free, we obtain that $\mathbb{Z}^w \otimes_{\mathbb{Z}\pi} \Gamma(A)$ can only
  have $2$-torsion. But by an argument involving the transfer map
  similar to the one in the proof of \cref{weird-ideal-B-inj} there
  can only be $p$-torsion, hence in this case there is actually no
  torsion.

  If $w$ is trivial, we have the sequence
  \[
    \cdots \to \pi \xrightarrow{\delta} (\mathbb{Z} \otimes_{\mathbb{Z}\pi} \Gamma(\mathbb{Z}\pi)) \oplus \mathbb{Z}
    \xrightarrow{\id \otimes f} \mathbb{Z} \otimes_{\mathbb{Z}\pi} \Gamma(A) \to \mathbb{Z} \to 0.
  \]
  Because of the $\mathbb{Z}$ term at the end of the sequence, all of the
  torsion in $\mathbb{Z}\otimes_{\mathbb{Z}\pi}\Gamma(A)$ is contained in the image of $\id
  \otimes_{\mathbb{Z}\pi} f$. By \cref{Gamma-zpi-free} we know that $\Gamma(\mathbb{Z}\pi)$ is a free
  $\mathbb{Z}\pi$-module, so $\mathbb{Z} \otimes_{\mathbb{Z}\pi} \Gamma(\mathbb{Z}\pi) \oplus \mathbb{Z}$ is torsion-free (see
  \cref{tors-free}). Hence, all of the torsion in $\mathbb{Z} \otimes_{\mathbb{Z}\pi} \Gamma(A)$
  must come from the kernel of $\id \otimes_{\mathbb{Z}\pi} f$, i.e.\ the image of
  $\delta$. Thus, we now calculate $\delta$.

  The map $\delta$ is the connecting map in the $\Tor$ long exact sequence,
  i.e.\ it is given by taking a resolution of the trivial right module
  $\mathbb{Z}$ and tensoring with each of the modules in our short exact
  sequence, then applying the snake lemma. We use the resolution of
  $\mathbb{Z}$ obtained from the one in the proof of \cref{hfour} by replacing
  $t$ with $t^{-1}$ everywhere. Since $w$ is trivial, $\ol t =
  t^{-1}$ and we will write $\ol t$ for clarity.
  
  Note that $\delta$ must be trivial on $\mathbb{Z}/p \leq \pi$, since $\mathbb{Z} \otimes_{\mathbb{Z}\pi} \Gamma(\mathbb{Z}\pi) \oplus
  \mathbb{Z}$ is torsion-free, so we will restrict the map $\partial_1$ to the summand
  of $C_1$ corresponding to the $\mathbb{Z}$ summand of $\pi$. We have the
  commutative diagram
  \begin{equation*}
    \begin{tikzcd}
      \Gamma(\mathbb{Z}\pi) \oplus \mathbb{Z}\pi \arrow[r, "f"] \arrow[d, "1-\ol t"]
      & \Gamma(A) \arrow[r] \arrow[d, "1-\ol t"]
      & \mathbb{Z} \arrow[d, "0"] \\
      \Gamma(\mathbb{Z}\pi) \oplus \mathbb{Z}\pi \arrow[r, "f"]
      & \Gamma(A) \arrow[r]
      & \mathbb{Z}.
    \end{tikzcd}
  \end{equation*}
  The element $1 \in \mathbb{Z}$ lifts to $N \otimes N \in \Gamma(A)$ on the top row,
  which then descends to
  \begin{align*}
    (1-\ol t)(N \otimes N) & = N \otimes N - \ol tN \otimes\ol tN \\
                          & = N \otimes N
                            -(1-(1-\ol t))N\otimes(1-(1-\ol t))N
                            \\ 
                          & = -(1-\ol t)N\otimes(1-\ol t)N
                            +(1-\ol t)N \sq N 
  \end{align*}
  in $\Gamma(A)$ on the bottom row. Now, after noting that $N$ commutes
  with $\ol t$, we can see that the RHS is equal to the image of
  $(-N \otimes N, N) \in \Gamma(\mathbb{Z}\pi) \oplus \mathbb{Z}\pi$ under $f$. By \cref{Gamma-zpi-free} we
  can express $N \otimes N \in \Gamma(\mathbb{Z}\pi)$ in terms of basis elements as follows:
  \begin{align*}
    N \otimes N & = N(1\otimes1) + N(1 \sq T) + \cdots + N\left(1 \sq
            T^{\frac{p-1}{2}}\right) \\
          & = Nu,
  \end{align*}
  where $u \coloneqq 1\otimes1 + 1 \sq T + \cdots + 1 \sq T^{\frac{p-1}{2}}$. This
  shows that after tensoring $(-N \otimes N, N) \in \Gamma(\mathbb{Z}\pi) \oplus \mathbb{Z}\pi$ becomes $(-1 \otimes
  pu, p) \in \mathbb{Z} \otimes_{\mathbb{Z}\pi} \Gamma(\mathbb{Z}\pi) \oplus \mathbb{Z}$. Hence $(-1 \otimes u, 1)$ maps to an element
  of order $p$ in $\mathbb{Z} \otimes_{\mathbb{Z}\pi} \Gamma(A)$. This completes the proof of the
  lemma.
\end{proof}

\begin{lemma}
  \label{torsfin}
  Suppose $q\neq1$. If $q^2 \not\equiv w(t) \mod p$, the group $\mathbb{Z}^w\otimes_{\mathbb{Z}\pi}\Gamma(A)$
  is torsion-free; otherwise, its torsion subgroup is isomorphic to
  $\mathbb{Z}/p$ and is generated by the image of the element
  \[
    \rho = N\sq E - E\otimes E - E\sq TE - \cdots - E\sq T^{\frac{p-1}{2}}E \in \Gamma(A),
  \]
  where $E = \left((\sum_{i=0}^{p-2}q^{p-2-i}\ol t^i)(q-\ol
  t) + \frac{1-q^{p-1}}{p}N\right)$.
\end{lemma}

\begin{proof}
  Let $K$ be a presentation complex for the group $\pi$. Then we have
  that $H^2(K; \mathbb{Z}\pi) \cong A$. The structure of $\Gamma(A)$ as a
  $\mathbb{Z}\mathbb{Z}$-module is induced by restricting the $\mathbb{Z}\pi$-module structure of
  $H^2(K; \mathbb{Z}\pi) \cong H^2(\wh K; \mathbb{Z}\mathbb{Z})$, where $\wh K$ is the cover
  of $K$ corresponding to the $\mathbb{Z}$ subgroup, and the isomorphism on
  cohomology modules holds because this cover is a finite complex. Now
  an argument involving the transfer map similar to the one in the
  proof of \cref{weird-ideal-B-inj} shows that $\mathbb{Z}^w\otimes_{\mathbb{Z}\pi}\Gamma(A)$ can
  only contain $p$-torsion.
  
  Again as in that proof, consider the subgroup $\pi' = \langle T, t^{p-1}\rangle$
  of index $p-1$ isomorphic to $\mathbb{Z}/p \times \mathbb{Z}$ and the cover $\ol K$
  of $K$ corresponding to this subgroup. By \cref{torsfin-dir}, we
  have that, for $J \coloneqq H^2(\ol K; \mathbb{Z}\pi')$, there is one
  copy of $\mathbb{Z}/p$ in $\mathbb{Z}\otimes_{\mathbb{Z}\pi'}\Gamma(J)$ and any $p$-torsion element in
  $\mathbb{Z}\otimes_{\mathbb{Z}\pi}\Gamma(A)$ must map to this by the transfer map since it must
  survive the composition given by multiplication by $p-1$. This means
  there must be at most one copy of $\mathbb{Z}/p$ in $\mathbb{Z}\otimes_{\mathbb{Z}\pi}\Gamma(A)$.

  Consider the $\mathbb{Z}\pi$-module map $\mathbb{Z}\pi \to A$ sending $1 \mapsto q-\ol
  t$. This map is injective and has cokernel $\mathbb{Z}[1/q]$, where $t$ acts
  by $w(t)1/q$ and $T$ acts trivially. By \cref{exact-seqs} we have
  the associated short exact sequences
  \[
    0 \to \Gamma(\mathbb{Z}\pi) \to \Gamma(A) \to D \to 0
  \]
  \[
    0 \to \mathbb{Z}\pi\otimes\mathbb{Z}[1/q] \cong \mathbb{Z}[1/q]\pi \to D \to \Gamma(\mathbb{Z}[1/q]) \cong \mathbb{Z}[1/q^2] \to 0.
  \]
  Here $\mathbb{Z}[1/q^2]$ is the module where $t$ acts by $1/q^2$.

  Using the free resolution for the right module $\mathbb{Z}$ from
  \cref{hfour}, we can calculate directly that $H_1(\pi; \mathbb{Z}[1/q^2])$ is
  torsion and $H_1(\pi; \mathbb{Z}[1/q]\pi^w)$ is trivial, hence by the homology
  long exact sequence associated to the latter of the above sequences
  we can conclude that $H_1(\pi; D^w)$ is torsion. We get the two short
  exact sequences
  \[
    0 \to \mathbb{Z}^w\otimes_{\mathbb{Z}\pi}\Gamma(\mathbb{Z}\pi) \to \mathbb{Z}^w\otimes_{\mathbb{Z}\pi}\Gamma(A) \to \mathbb{Z}^w\otimes_{\mathbb{Z}\pi}D \to 0
  \]
  \[
    0 \to \mathbb{Z}[1/q] \to \mathbb{Z}^w\otimes_{\mathbb{Z}\pi}D \to \mathbb{Z}^w\otimes_{\mathbb{Z}\pi}\Gamma(\mathbb{Z}[1/q]) \cong \mathbb{Z}/(q^2-w(t)) \to 0.
  \]
  The torsion subgroup of $\mathbb{Z}^w\otimes_{\mathbb{Z}\pi}\Gamma(A)$, which we already saw is at
  most $\mathbb{Z}/p$, must inject into $\mathbb{Z}^w\otimes_{\mathbb{Z}\pi}D$, and the torsion subgroup
  there must inject into $\mathbb{Z}/(q^2-w(t))$. We conclude that there is no
  torsion unless $q^2 \equiv w(t) \mod p$.

  Assume then that $q^2 \equiv w(t) \mod p$. Considering again the subgroup
  $\pi'$. Let $J$ be the ideal $\langle1-t^{p-1}, N\rangle$ in $\mathbb{Z}\pi'$ and consider
  the following map of $\mathbb{Z}\pi'$-modules:
  \begin{align*}
    J        & \to A \\
    1-t^{p-1} & \mapsto E \\
    N        & \mapsto N,
  \end{align*}
  where $E$ is the expression from the statement of the lemma. This
  map induces a map $\mathbb{Z}\otimes_{\mathbb{Z}\pi'}\Gamma(J) \to \mathbb{Z}^w\otimes_{\mathbb{Z}\pi}\Gamma(A)$ under which the
  torsion element from \cref{torsfin-dir} maps to the candidate
  torsion element $\rho$. We need to show this element is nontrivial.

  For this, consider the map $A \to (\mathbb{Z}/p)^q$ which evaluates as $1$ at
  $N$ and as $0$ at $q-\ol t$. This map induces a map
  $\mathbb{Z}^w\otimes_{\mathbb{Z}\pi}\Gamma(A) \to \mathbb{Z}^w\otimes_{\mathbb{Z}\pi}\Gamma((\mathbb{Z}/p)^q) \cong \mathbb{Z}/p$ under which $\rho$ maps to
  $2\frac{1-q^{p-1}}{p}$. So it remains to show that $1-q^{p-1} \not\equiv
  0 \mod p^2$.

  Let $k = (p-1)/2$, then $q^{p-1} = (q^2)^k = (w(t)+ap)^k$ for some
  positive $a < p$ (since $1 < q < p$). Note that $q^2 \equiv -1 \mod p$
  can only occur if $k$ is even, hence $w(t)^k = 1$. Then $q^{p-1} =
  (w(t)+ap)^k = 1 + w(t)akp \mod p^2$, but $p$ does not divide
  $w(t)ak$, therefore $q^{p-1} \not\equiv 1 \mod p^2$ as required.
\end{proof}

\begin{theorem}
  \label{p-htpy-types}
  Let $p$ be an odd prime, and let $M$ be a Poincar\'e $4$-complex
  with fundamental group $\mathbb{Z}/p\rtimes\mathbb{Z}$, orientation character $w$, and
  $\pi_2(M)$ stably isomorphic to $A\oplus A'$. If $q^2 \not\equiv w(t) \mod p$,
  $M$ is determined up to homotopy by its quadratic $2$-type;
  otherwise, there are $p$ different homotopy types over the Postnikov
  $2$-type $B$ of Poincar\'e $4$-complexes with the same quadratic
  $2$-type as $M$.
\end{theorem}

\begin{proof}
  We will apply \cref{realization-cor}. The hypothesis of that
  corollary is satisfied by \cref{flourish}, whose hypotheses in turn
  we showed to hold for $\mathbb{Z}/p\rtimes\mathbb{Z}$ in \cref{dualiso,histar-cyclic}. The
  map $\ev^{*}$ is injective by \cite[Lemma~5.1]{KPR}, and the kernel
  of $\mathcal{B}_{\pi_2(B)}$ is isomorphic to the sum of the kernels of the maps
  $\mathcal{B}_A$, $\mathcal{B}_{A'}$, and $A\otimes_{\mathbb{Z}\pi} A' \to \Hom_{\mathbb{Z}\pi}(A^{**}, A^{**})$ by
  \cref{dsum-rem}. The last of these coincides with the kernel of
  $\varphi_B$ by \cref{ker-phi-cor}. Hence we have that the the set of
  homotopy types over $B$ has the same cardinality as
  $\ker\mathcal{B}_A\oplus\ker\mathcal{B}_{A'}$. The theorem then follows by \cref{rkzero} and
  the calculations in \cref{weird-ideal-B-inj,torsfin-dir,torsfin}.
\end{proof}

\subsection{The form $\lambda_M^{\mathbb{Z}/p}$ detects the homotopy types}

\begin{theorem}
  \label{mainall}
  Let $p$ be an odd prime, and let $M$ and $M'$ be Poincar\'e
  $4$-complexes with fundamental group $\mathbb{Z}/p\rtimes\mathbb{Z}$ and isomorphic
  quadratic $2$-types. Suppose that $\pi_2(M) \cong \pi_2(M')$ is stably
  isomorphic to $A\oplus A'$. If $\lambda_M^{\mathbb{Z}/p} \cong \lambda_{M'}^{\mathbb{Z}/p}$, then $M$ and
  $M'$ are homotopy equivalent.
\end{theorem}

\begin{proof}
  If $B$ is the common Postnikov $2$-type of $M$ and $M'$, let $B^s$
  be the Postnikov $2$-type of $B \vee S^2$. The inclusion and collapse
  maps induce maps $B \to B^s \to B$ whose composition is the identity, so
  the map $H_4(B; \mathbb{Z}^w) \to H_4(B^s; \mathbb{Z}^w)$ is an inclusion. The subgroup
  of torsion elements where $f_{*}[M]-f_{*}[M']$ lives maps
  isomorphically to the subgroup where live the differences of
  fundamental classes of Poincar\'e complexes $X$ with $\pi_2(X) \cong
  \pi_2(M)\oplus\mathbb{Z}\pi$. Therefore the forms $f_{*}\lambda_M^{\mathbb{Z}/p}, f_{*}'\lambda_{M'}^{\mathbb{Z}/p}$
  induce isomorphic forms on $H^2(B^s; \mathbb{Z}\pi)$ which are also
  respectively isomorphic to the forms $f_{*}\lambda_{X}^{\mathbb{Z}/p},
  f_{*}'\lambda_{X'}^{\mathbb{Z}/p}$ for some $X$ and $X'$. Hence by induction we can
  assume that $\pi_2(M) \cong A\oplus A'\oplus F$ for some free module $F$.

  By \cref{p-htpy-types}, if $q^2 \not\equiv w(t) \mod p$, we already have
  a homotopy equivalence $M \simeq M'$. Thus, assume $q^2 \equiv w(t) \mod p$
  and that $M$ and $M'$ are homotopy inequivalent over the Postnikov
  $2$-type $B$. Write $\phi \colon \wt B \to B$ for the universal
  covering map; then $\phi_{*} \colon H_4(\wt B; \mathbb{Z}) \to H_4(B; \mathbb{Z}^w)$
  factors through $\varphi_B$, and for some positive $k < p$ we have that
  $f_{*}[M]-f_{*}'[M'] = k\phi_{*}\rho \in H_4(B; \mathbb{Z}^w)$, where $\rho \in \Gamma(A) \subseteq
  H_4(\wt B; \mathbb{Z}) \cong \Gamma(A\oplus A'\oplus F)$ is the element from
  \cref{torsfin-dir} or \cref{torsfin}, depending on whether $q=1$. We
  have to show that in this case we cannot have $\lambda_M^{\mathbb{Z}/p} \cong
  \lambda_{M'}^{\mathbb{Z}/p}$. Hence we need a pair $a,b \in H^2(B; (\mathbb{Z}/p)^q)$ such
  that $b\cap(a\cap f_{*}[M]) \neq b\cap(a\cap f_{*}'[M'])$, i.e.\ such that $b\cap(a\cap
  k\phi_{*}\rho) = \phi^{*}b\cap(\phi^{*}a\cap k\rho) \neq 0 \in \mathbb{Z}/p$. As this is valued in
  $\mathbb{Z}/p$, we do not lose any generality by assuming $k=1$.

  Consider the Serre spectral sequence with local coefficients of the
  fibration $\wt B \to B \to B\pi$. This has $E_2$ page $H^i(\pi; H^j(\wt B;
  \mathbb{Z}/p)^q)$ and converges to $H^{i+j}(B; (\mathbb{Z}/p)^q)$. Since $B$ has no
  $3$-cells, the differential $d^3 \colon H^0(\pi; H^2(\wt B; \mathbb{Z}/p)^q) \cong
  \Hom_{\mathbb{Z}\pi}(\pi_2(B), (\mathbb{Z}/p)^q) \to H^3(\pi; (\mathbb{Z}/p)^q)$ must be
  surjective. The same analysis applies for the fibration $K(A', 2)
  \to K \to B\pi$, where $K$ is the presentation complex for the group $\pi$,
  and we have a map $K \to B$ which is an isomorphism on fundamental
  groups, hence by naturality of the spectral sequence we obtain the
  following commutative diagram with exact rows:
  \begin{equation*}
    \begin{tikzcd}[column sep=small]
      H^2(B; (\mathbb{Z}/p)^q) \arrow[r, "\phi^{*}"] \arrow[d]
      & \Hom_{\mathbb{Z}\pi}(A\oplus A'\oplus F, (\mathbb{Z}/p)^q) \arrow[r] \arrow[d]
      & H^3(\pi; (\mathbb{Z}/p)^q) \arrow[r] \arrow[d, equals]
      & 0 \\
      H^2(K; (\mathbb{Z}/p)^q) \arrow[r]
      & \Hom_{\mathbb{Z}\pi}(A', (\mathbb{Z}/p)^q) \arrow[r]
      & H^3(\pi; (\mathbb{Z}/p)^q) \arrow[r]
      & 0.
    \end{tikzcd}
  \end{equation*}
  The middle map is given by restriction to the $A'$ summand, so the
  diagram shows that a map $\alpha \in \Hom_{\mathbb{Z}\pi}(A\oplus A'\oplus F)$ that is trivial
  on the $A'$ summand is in the image of $\phi^{*}$.

  Therefore to finish the proof, we need to exhibit $\pi$-linear maps
  $\alpha,\beta \colon A\oplus A'\oplus F \to (\mathbb{Z}/p)^q$ that are trivial on the $A'$
  summand and for which $\beta\cap(\alpha\cap\rho) \neq 0$. The choice of $\alpha,\beta$
  will be different for the different forms of $\rho$, i.e.\ depending on
  whether $q=1$. We will verify that $\beta\cap(\alpha\cap\rho) \neq 0$ using
  \cref{whitehead-eval}. That lemma requires that $\pi_2(B)$ is free as
  an abelian group, which is true in our case since $A$ and $A'$ are
  both submodules of $\mathbb{Z}\pi$.

  For $q=1$, pick $\alpha$ to be the $\pi$-linear map which evaluates as $1$
  at $N \in A$ and $0$ elsewhere, and $\beta$ to be the map which evaluates
  as $1$ at $1-\ol t \in A$ and $0$ elsewhere. Since $\rho$ is the element
  from \cref{torsfin-dir}, by \cref{whitehead-eval} we have that
  \begin{align*}
    (\beta\cap(\alpha\cap\rho) =~ & \alpha(N)\beta(1-\ol t) + \alpha(1-\ol t)\beta(N)\\
    & - \alpha(1-\ol t)\beta(1-\ol t) - \cdots - \alpha(1-\ol
      t)\beta(T^{\frac{p-1}{2}}(1-\ol t)) -
      \alpha(T^{\frac{p-1}{2}}(1-\ol t))\beta(1-\ol t).
  \end{align*}
  The first summand is $1$ and the rest are $0$, so $\beta\cap(\alpha\cap\rho) = 1$ is
  nonzero in this case.

  For $q\neq1$, pick both $\alpha$ and $\beta$ to be the map which evaluates as
  $1$ at $N \in A$ and $0$ elsewhere. This is the map we already used at
  the end of the proof of \cref{torsfin}, and we saw that the image of
  $\rho$ survives in the image of the induced map $\mathbb{Z}^w\otimes_{\mathbb{Z}\pi}\Gamma(A) \to
  \mathbb{Z}^w\otimes_{\mathbb{Z}\pi}\Gamma((\mathbb{Z}/p)^q) \cong \mathbb{Z}/p$. But the map $\beta\cap(\alpha\cap-) \colon \Gamma(A) \to \mathbb{Z}/p$
  factors through this map by \cref{whitehead-eval}, so $\beta\cap(\alpha\cap\rho)$ is
  nonzero in this case also, and the proof is complete.
\end{proof}

\bibliographystyle{amsalpha} \bibliography{htpytypes}
\end{document}